\documentclass[11pt,twoside,reqno,a4paper]{amsart}

\usepackage{microtype}
\usepackage{cite}
\usepackage[OT1]{fontenc}
\usepackage{lmodern}
\usepackage{type1cm}
\usepackage{mathtools}
\usepackage{amssymb}
\usepackage[shortlabels]{enumitem}
\IfFormatAtLeastTF{2026-06-01}{}{\usepackage{aliascnt}}
\usepackage[a4paper,margin=1in]{geometry}
\usepackage{xcolor}
\usepackage{graphicx}
\graphicspath{{graphics/}}
\usepackage{hyperref}
\hypersetup{colorlinks=true,linkcolor=black,citecolor=black,urlcolor=black}
\usepackage[capitalize,noabbrev,nameinlink]{cleveref}

\numberwithin{equation}{section}
\allowdisplaybreaks

\makeatletter
\renewcommand{\paragraph}{\@startsection{paragraph}{4}%
  \z@\medskipamount{-\fontdimen2\font}%
  {\normalfont\bfseries}}
\makeatother

\theoremstyle{plain}
\newtheorem{theorem}{Theorem}[section]
\makeatletter
\IfFormatAtLeastTF{2026-06-01}{%
  \newtheorem{proposition}[theorem]{Proposition}
  \newtheorem{lemma}[theorem]{Lemma}
  \newtheorem{corollary}[theorem]{Corollary}
  \theoremstyle{definition}
  \newtheorem{definition}[theorem]{Definition}
  \theoremstyle{remark}
  \newtheorem{remark}[theorem]{Remark}
  \let\cref@orig@resetby\cref@resetby
  \def\cref@resetby#1#2{%
    \@ifundefined{alias@ctr@#1}{\cref@orig@resetby{#1}{#2}}{%
      \edef\cref@alias@target{\csname alias@ctr@#1\endcsname}%
      \expandafter\cref@orig@resetby\expandafter{\cref@alias@target}{#2}}}%
}{%
  \newaliascnt{proposition}{theorem}
  \newtheorem{proposition}[proposition]{Proposition}
  \aliascntresetthe{proposition}
  \newaliascnt{lemma}{theorem}
  \newtheorem{lemma}[lemma]{Lemma}
  \aliascntresetthe{lemma}
  \newaliascnt{corollary}{theorem}
  \newtheorem{corollary}[corollary]{Corollary}
  \aliascntresetthe{corollary}
  \theoremstyle{definition}
  \newaliascnt{definition}{theorem}
  
  \aliascntresetthe{definition}
  \theoremstyle{remark}
  \newaliascnt{remark}{theorem}
  
  \aliascntresetthe{remark}%
}
\makeatother

\crefname{proposition}{Proposition}{Propositions}
\Crefname{proposition}{Proposition}{Propositions}
\crefname{lemma}{Lemma}{Lemmas}
\Crefname{lemma}{Lemma}{Lemmas}
\crefname{corollary}{Corollary}{Corollaries}
\Crefname{corollary}{Corollary}{Corollaries}
\crefname{definition}{Definition}{Definitions}
\Crefname{definition}{Definition}{Definitions}
\crefname{remark}{Remark}{Remarks}
\Crefname{remark}{Remark}{Remarks}

\newcommand{\bbR}{\mathbb{R}}
\newcommand{\bbQ}{\mathbb{Q}}
\newcommand{\bbZ}{\mathbb{Z}}
\newcommand{\bbC}{\mathbb{C}}
\newcommand{\bbE}{\mathbb{E}}
\newcommand{\bbP}{\mathbb{P}}
\newcommand{\cC}{\mathcal{C}}
\newcommand{\cA}{\mathcal{A}}
\newcommand{\cE}{\mathcal{E}}
\newcommand{\cK}{\mathcal{K}}
\newcommand{\cP}{\mathcal{P}}

\newcommand{\cL}{\mathcal{L}}
\newcommand{\cM}{\mathcal{M}}
\newcommand{\cT}{\mathcal{T}}
\newcommand{\cW}{\mathcal{W}}
\newcommand{\eps}{\varepsilon}
\newcommand{\dd}{\,\mathrm{d}}

\renewcommand{\leq}{\leqslant}
\renewcommand{\geq}{\geqslant}

\DeclarePairedDelimiter{\abs}{\lvert}{\rvert}
\DeclarePairedDelimiter{\norm}{\lVert}{\rVert}
\DeclarePairedDelimiter{\set}{\lbrace}{\rbrace}

\DeclareMathOperator{\dist}{dist}
\DeclareMathOperator{\diam}{diam}
\DeclareMathOperator{\Lip}{Lip}
\DeclareMathOperator{\Law}{Law}
\DeclareMathOperator{\Var}{Var}

\newcommand{\BM}{W}
\newcommand{\KA}{K_A}
\newcommand{\KB}{K_B}
\newcommand{\muA}{\mu_A}
\newcommand{\muB}{\mu_B}

\begin{document}

\title{Distinguishing sets and measures from their planar Brownian images}

\author{Sascha Troscheit}
\address{Department of Mathematics, Box 524, 751 20 Uppsala University, Sweden}
\email{sascha.troscheit@math.uu.se}

\begin{abstract}
Motivated by a question of Hochman, we investigate whether two deterministic time measures or time sets
can be distinguished almost surely from a single planar Brownian image. Using rescaled
correlation integrals, we give a criterion for mutual singularity of the laws of the image
measures, even when the time measures have the same dimension. For self-similar time sets 
satisfying the open set condition, we also recover the natural
occupation measure from the compact image. Together, these results allow us to
distinguish lattice and non-lattice self-similar sets and their natural measures within this class.
For a naturally parametrised family with two contractions, with fixed dimension, 
each parameter can be distinguished from all but finitely many others. If the dimension is
sufficiently small, every distinct parameter can be distinguished.
The results are formalised in Lean.
\end{abstract}

\maketitle

\section{Introduction}

Let $\BM=(\BM_t)_{t\in[0,1]}$ be standard Brownian motion in $\bbR^2$. If $\mu$ is a Borel
probability measure with support $K\subseteq[0,1]$, then $\nu=\BM_*\mu$ is a random probability
measure on $\bbR^2$, while $E=\BM(K)$ is a random compact set.
Given two prescribed deterministic time measures, can we identify which one produced a single
observed Brownian pushforward, with probability one? Equivalently, are their corresponding laws
mutually singular? Further, is it possible to distinguish between different self-similar sets
$K_A,K_B$ of the line, purely by almost sure images under $W$?
The Hausdorff dimension of these objects doubles and can be used to crudely distinguish between
different sets. However, when $\mu_A,\mu_B$ or $K_A,K_B$ are different measures or self-similar sets of
the same dimension, these questions are much more intricate.
This `inverse problem' was posed to the author by M.~Hochman\footnote{M.~Hochman asked whether one
  can distinguish the planar Brownian images of two self-similar sets, one
  homogeneous and one inhomogeneous. He posed the question to the author at the \emph{Geometry of
  Deterministic and Random Fractals} conference at the R\'enyi Institute in 2022 and more generally in the
  problem session at the \emph{Smooth Dynamics, Random Walks, and Fractal Geometry} workshop at the Institut Mittag-Leffler in June 2026.} and we resolve it here for a
large class of measures and sets.

For a probability measure $\nu$ on $\bbR^2$, write
\begin{equation}\label{eq:correlation-integral}
  \cC_r(\nu)=(\nu\times\nu)\set{(x,y)\in\bbR^2\times\bbR^2\colon\abs{x-y}<r}
\end{equation}
for its \emph{correlation integral}.
Equivalently, $\cC_r(\nu)=\int\nu(B(x,r))\dd\nu(x)$, with $B(x,r)$ denoting
the open ball in $\bbR^2$.

\subsection{Distinguishing image measures}
Let $0<s<1$. 
For an $s$-Frostman probability measure $\mu$ on $[0,1]$, we define the
\emph{normalised expected correlation profile}, or \emph{correlation profile} for short,
\begin{equation}\label{eq:expected-profile}
  H_\mu^s(t)=e^{2st}\bbE \cC_{e^{-t}}(\BM_*\mu),
  \qquad t\geq0,
\end{equation}
where the expectation is taken with respect to the law of the Brownian motion.
We show

\begin{theorem}\label{thm:main}
Let $0<s<1$, and let $\mu_1$ and $\mu_2$ be $s$-Frostman probability measures on $[0,1]$.
If
\begin{equation}\label{eq:profile-separation}
  \limsup_{t\to\infty}\abs{H_{\mu_1}^s(t)-H_{\mu_2}^s(t)}>0,
\end{equation}
then
\[
  \Law(\BM_*\mu_1)\perp\Law(\BM_*\mu_2)
  \quad\text{on }\cP(\bbR^2).
\]
\end{theorem}
Thus \cref{thm:main} allows us to distinguish almost surely from one measured image, but it is not a general
reconstruction theorem for an arbitrary source measure. 

For a fixed measure $\mu$, there is at most one exponent $s$ for which $H_\mu^s(t)$ remains bounded
above and bounded away from zero as $t\to\infty$. For an $s$-Ahlfors regular measure, this holds
at its regularity exponent. Measures with different `optimal' Frostman exponents can therefore
already be distinguished from a single measured Brownian image.
For two Ahlfors regular time measures with a common exponent $s$, Falconer and
Xiao~\cite[Corollary~2.2]{FalconerXiao2014} show that the generalised $L^q$-dimensions of the
Brownian image measures are almost surely $2s$ for every $q>1$.
Both compact Brownian images are also Salem sets of dimension $2s$ almost surely, by Kahane's theorem on Brownian
images~\cite[Chapter~17]{Kahane1985}, see also \cite[Chapter~12]{Mattila2015}.

For an $s$-Ahlfors regular time measure, $2s$ is thus the exact correlation dimension of its
Brownian image measure:
\[
  D_2(\BM_*\mu)=\lim_{r\downarrow0}\frac{\log\cC_r(\BM_*\mu)}{\log r}=2s
  \qquad\text{almost surely},
\]
see Sauer and Yorke~\cite[Proposition~2.3]{SauerYorke1997}.
We study the exact scaling behaviour of $r^{-s}\cC_r(\mu)$ at the correlation dimension. This
is analogous to studying local densities rather than local dimensions of measures. There is a
further link to local dimension:
Writing 
\[
  r^{-2s}\cC_r(\nu)=\int r^{-2s}\nu(B(x,r))\dd\nu(x)
\]
shows that the normalised correlation
integral is the expectation of local density ratios. Studying these local behaviour reveals more
about the original measure, and allows us to differentiate between different ones.

\paragraph{Self-similar sets.}
A self-similar iterated function system (IFS) on $[0,1]$ is a finite family $\{S_i\}_{i=1}^N$ of
contracting similarities of $\bbR$ that map $[0,1]$ into itself.
It is well-known that there exists a unique non-empty compact invariant set $K$ satisfying
$K=\bigcup_i S_i(K)$. Similarly, for every probability vector $(p_i)$ there exists an invariant
probability measure $\mu=\sum_i p_i(S_i)_*\mu$. We refer to \cite{Falconer1997} for more
details. The IFS satisfies the \emph{open set condition} if there is a non-empty bounded open
set $U\subset\bbR$ such that $S_i(U)\subseteq U$ for every $i$ and the sets $S_i(U)$ are
pairwise disjoint. The open set need not be an interval, and the intervals $S_i([0,1])$ may
overlap. By Schief~\cite{Schief1994}, this is equivalent to the \emph{strong open set condition} and
we may choose such a set $U$ with $U\cap K\neq\varnothing$.

The contraction ratios $r_i = |S_i'(0)|$ determine whether the IFS is of \emph{lattice type} or
\emph{non-lattice type}. In the lattice case, the additive subgroup generated by
$\{\log(1/r_i)\}_{i=1}^N$ is $\ell\bbZ$ for some $\ell>0$ and
the ratios are integer powers of a common ratio $e^{-\ell}$. In the non-lattice
case, this subgroup is dense in $\bbR$. We use the same terminology for the associated
self-similar sets and measures as well.

Suppose that the IFS satisfies the open set condition and has similarity dimension
$s\in(0,1)$, determined by $\sum_i r_i^s=1$. Its natural self-similar probability measure
$\mu$ has weights $r_i^s$ and is $s$-Ahlfors regular~\cite{Falconer1997}, and hence
$s$-Frostman. At this dimension, its correlation profile has the following asymptotic behaviour.

\begin{theorem}\label{thm:profile-asymptotics}
Let $\mu$ be the natural self-similar probability measure of an IFS satisfying the open set
condition, with similarity dimension $s\in(0,1)$.
If the IFS is of lattice type with span $\ell$, then there is a continuous, strictly positive,
non-constant $\ell/2$-periodic function $\widetilde H_\mu^s$ such that
\[
  H_\mu^s(t)-\widetilde H_\mu^s(t)\to0
  \qquad\text{as }t\to\infty.
\]
If the IFS is of non-lattice type, then
\[
  H_\mu^s(t)\to C_\mu\in(0,\infty)
  \qquad\text{as }t\to\infty.
\]
\end{theorem}

Combining \cref{thm:main,thm:profile-asymptotics} gives the following consequence.

\begin{corollary}\label{thm:cantor-application}
Let $\muA$ and $\muB$ be the natural self-similar probability measures of two IFSs on $[0,1]$
satisfying the open set condition, with common similarity dimension $s\in(0,1)$. If the first
IFS is of lattice type and the second is of non-lattice type, then
\[
  \Law(\BM_*\muA)\perp\Law(\BM_*\muB).
\]
\end{corollary}

\paragraph{A two-contraction family.}
We can show substantially more for a family of self-similar measures with two contractions,
without requiring the systems to have different lattice types. Fix $0<s<1$. For $p\in(0,1)$,
put $q=1-p$, $\lambda_1=p^{1/s}$ and $\lambda_2=q^{1/s}$, and let $K_p$ be the attractor of
\[
  S_1(x)=\lambda_1x,
  \qquad S_2(x)=\lambda_2x+1-\lambda_2.
\]
We write $\mu_p$ for its natural self-similar probability measure with weights $p$ and $q$.
Since $\lambda_1+\lambda_2<p+q=1$, the IFS is strongly separated and $\mu_p$ is $s$-Ahlfors
regular. Reflection in $1/2$ interchanges $p$ and $1-p$ and preserves correlation integrals,
so we restrict comparisons to $p\in(0,1/2]$. The homogeneous system occurs at $p=1/2$.
\Cref{fig:two-time-sets} shows the Brownian images of two time sets of the family.

\begin{figure}[t]
  \centering
  \includegraphics[width=0.48\textwidth]{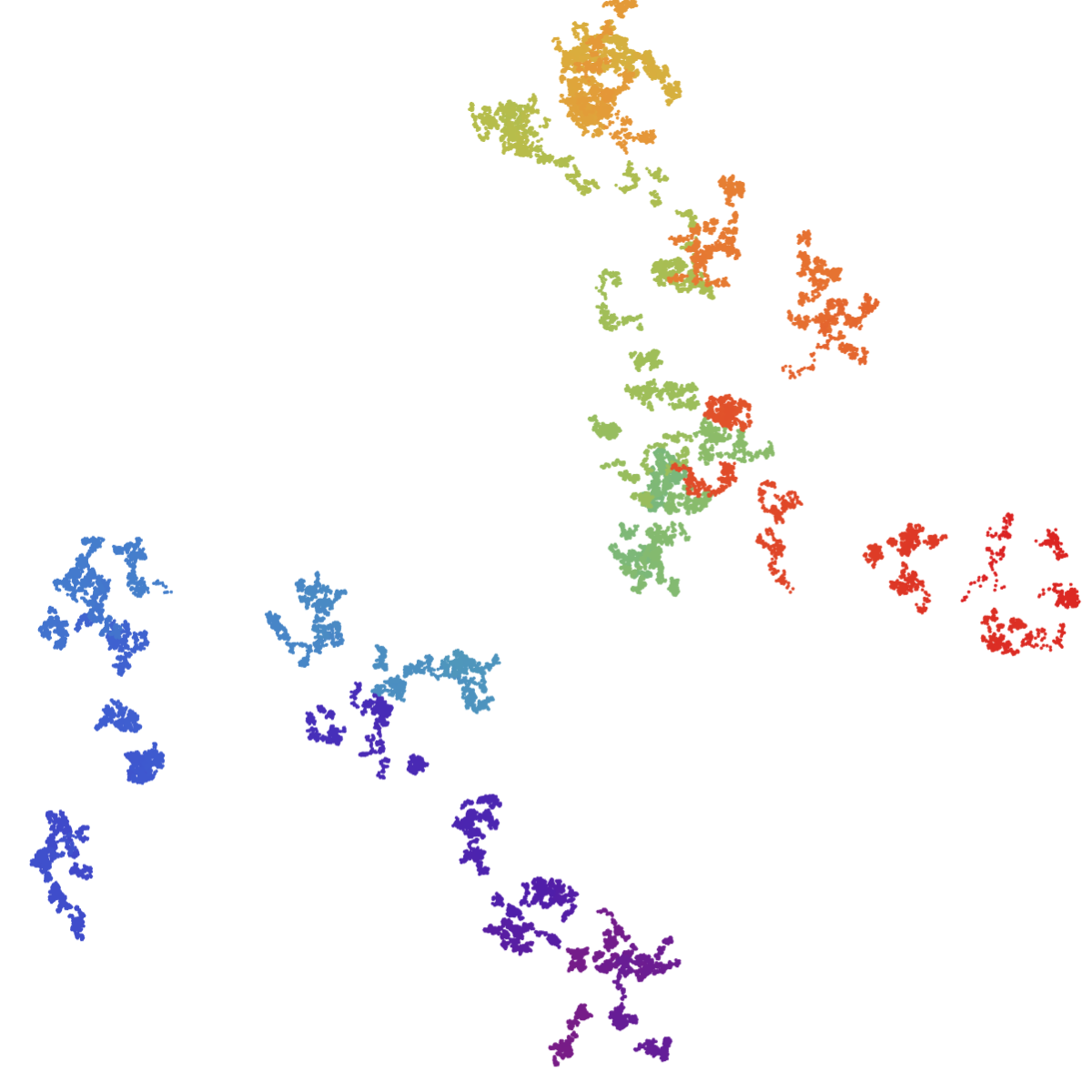}\hfill
  \includegraphics[width=0.48\textwidth]{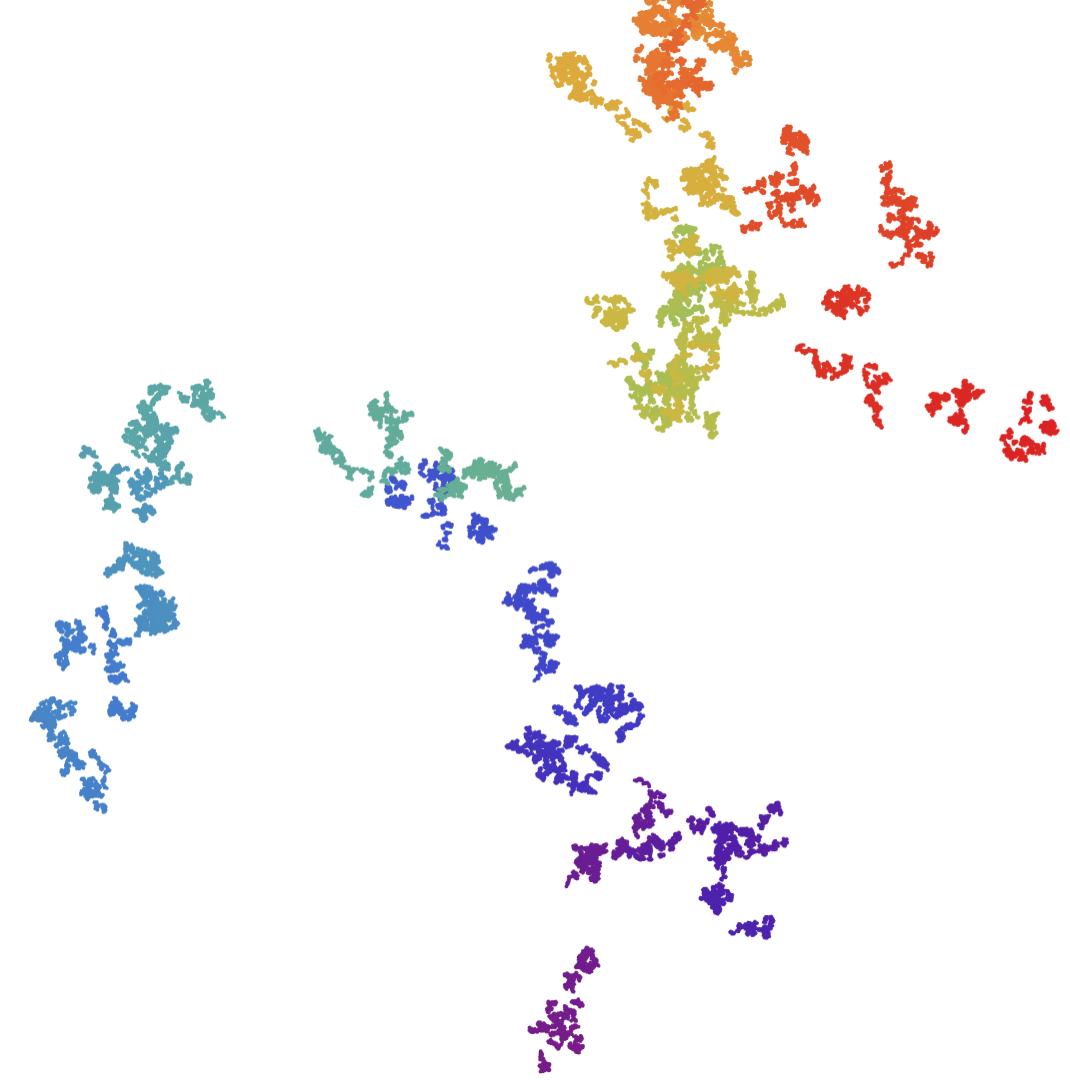}
  \caption{Brownian images $\BM(K_{p_1})$ (left) and $\BM(K_{p_2})$ (right) of two time sets of
  the two-contraction family, coloured by time, for $s=4/5$, $p_1=2/5$, and $p_2=1/2$.}
      \label{fig:two-time-sets}
\end{figure}

We write
\[
  \overline H_s(p)=\lim_{R\to\infty}\frac1R\int_0^R H_{\mu_p}^s(t)\dd t,
\]
for the mean correlation profile,
whose existence follows from \cref{thm:fixed-dimension-profiles}. By
\cref{thm:fixed-dimension-analyticity,thm:fixed-dimension-range}, it is positive and real
analytic on $(0,1)$ and tends to zero as $p\downarrow0$. We write $D_s$ for the set of its
critical points in $(0,1/2)$.

\begin{theorem}\label{thm:fixed-dimension-separation}
Let $s\in(0,1)$. For every $p_1\in(0,1/2]$, the set of $p_2\in(0,1/2]$ with
$\overline H_s(p_2)=\overline H_s(p_1)$ is finite. Consequently, for $p_2\in(0,1/2]$ not in this
exceptional set,
$\overline H_s(p_1)\neq\overline H_s(p_2)$, and
\[
  \Law(\BM_*\mu_{p_1})\perp\Law(\BM_*\mu_{p_2}).
\]
The set $D_s$ has no accumulation point in $(0,1/2]$, and the mean is strictly monotone on
every component of $(0,1/2)\setminus D_s$.
\end{theorem}

Moreover, for some $\eps_s\in(0,1/2)$, the mean is strictly monotone
on $(1/2-\eps_s,1/2)$, and the image-measure laws in this interval are pairwise mutually
singular. Numerical evaluation of \cref{eq:fixed-dimension-mean} suggests that the mean is
strictly increasing on $(0,1/2)$ for every $s\in(0,1)$, which would separate every pair of
distinct parameters. The stronger assertion that $D_s$ is empty is also supported numerically, but
we do not prove either full claim here.

However, we prove strict increase throughout the parameter range for all sufficiently small dimensions.

\begin{theorem}\label{thm:small-dimension-monotone}
There exists $s_0\in(0,1)$ such that, for every $s\in(0,s_0]$, the mean
$p\mapsto\overline H_s(p)$ is strictly increasing on $(0,1/2]$. Consequently, for all distinct
$p_1,p_2\in(0,1/2]$,
\[
  \Law(\BM_*\mu_{p_1})\perp\Law(\BM_*\mu_{p_2}).
\]
\end{theorem}

By the symmetry $\overline H_s(p)=\overline H_s(1-p)$, the mean then attains its maximum over
$(0,1)$ only at the homogeneous parameter $p=1/2$. For each fixed $0<s\leq1/2$, the estimates
also give a positive derivative near $p=0$. The remaining critical points lie in a compact
subinterval of $(0,1/2]$ and cannot accumulate by \cref{thm:fixed-dimension-separation}, giving
the following consequence.

\begin{corollary}\label{thm:critical-set-finite}
For every $s\in(0,1/2]$ there exists $p_0\in(0,1/2)$ such that $\overline H_s$ is strictly
increasing on $(0,p_0]$. In particular, the critical set $D_s$ is finite.
\end{corollary}

\paragraph{Missing digit set}
\cref{thm:main} also allows us to distinguish IFSs within the same classes, such as for IFSs with
the same number of maps with the same contraction ratios. In
\cref{thm:base-five-profiles}, we compare the base-five missing-digit sets with digit sets
$\{0,1,4\}$ and $\{0,2,4\}$, equipped with equal weights. Both have dimension
$\log3/\log5$ and the same logarithmic period, but their limiting normalised correlation profiles
differ.
One such difference is that the first does not satisfy the strong separation property (images of
$[0,1]$ meet at a point), whereas the second is strongly separated.

The complementary example in \cref{thm:homometric-example} has non-isometric attractors whose
natural measures have identical correlation integrals, so their expected correlation profiles do
not distinguish the Brownian image laws.

\subsection{Distinguishing image sets}
A measured image contains significantly more information than just the support of the measure.
Hence, \cref{thm:main} does not by itself
imply the corresponding statement for the compact image sets.
However, we may recover this
distribution from small neighbourhoods of the image for self-similar sets.

Let $\cK(\bbR^2)$ be the space of
non-empty compact subsets of $\bbR^2$, equipped with the Hausdorff metric. For
$F\in\cK(\bbR^2)$ and $r>0$,
put
\begin{equation}\label{eq:neighbourhood-measure-definition}
  M_r(F)=\int_{\bbR^2}\left(1-\frac{\dist(x,F)}r\right)_+\dd x,
  \qquad
  \dd\nu_{F,r}(x)=\frac{1}{M_r(F)}
  \left(1-\frac{\dist(x,F)}r\right)_+\dd x.
\end{equation}
Thus $\nu_{F,r}$ is a probability measure determined by the compact set $F$.

\begin{theorem}\label{thm:minkowski-reconstruction}
Let $\{S_i\}_{i=1}^N$ be a self-similar IFS on $[0,1]$ satisfying the open set condition,
with contraction ratios $r_i$ and attractor $K$. Assume that its
similarity dimension $s$ lies in $(0,1)$, and let $\mu$ be the self-similar probability measure
with weights $r_i^s$. Then
\begin{equation}\label{eq:minkowski-reconstruction}
  \nu_{\BM(K),e^{-n}}\to \BM_*\mu
  \quad\text{weakly almost surely as }n\to\infty.
\end{equation}
In particular, $\BM_*\mu$ is almost surely a Borel function of the compact set $\BM(K)$.
\end{theorem}

\Cref{thm:minkowski-reconstruction} shows that the compact image almost surely determines
its natural occupation measure. For measure laws, we already have the lattice/non-lattice
separation in \cref{thm:cantor-application}, so reconstruction gives the following corollary.
\begin{corollary}\label{thm:cantor-set-application}
Let $\KA$ and $\KB$ be the attractors of two IFSs on $[0,1]$ satisfying the open set condition,
with common similarity dimension $s\in(0,1)$. If the first IFS is of lattice type and the
second is of non-lattice type, then
\[
  \Law(\BM(\KA))\perp\Law(\BM(\KB))
  \quad\text{on }\cK(\bbR^2).
\]
\end{corollary}

The same reconstruction transfers the two-contraction results to compact image sets.

\begin{corollary}
For every $s\in(0,1)$ and all $p_1,p_2\in(0,1/2]$ with
$\overline H_s(p_1)\neq\overline H_s(p_2)$,
\[
  \Law(\BM(K_{p_1}))\perp\Law(\BM(K_{p_2}))
  \quad\text{on }\cK(\bbR^2).
\]
For each fixed $p_1$, this holds for all but finitely many $p_2\in(0,1/2]$. If $s\leq s_0$,
where $s_0$ is as in \cref{thm:small-dimension-monotone}, it holds for every $p_1\neq p_2$.
\end{corollary}

The mutual singularity of the laws of the images of the sets is equivalent to the existence of
a Borel set $\cA\subset\cK(\bbR^2)$ such that
\[
  \bbP(\BM(\KA)\in\cA)=1,
  \qquad
  \bbP(\BM(\KB)\in\cA)=0,
\]
and a single image set allows us, almost surely, to decide which of the original time sets were
given.

\subsection{Related literature}
Related reconstruction questions have appeared for other random images. Pemantle, Peres, Pitman
and Yor~\cite{PemantlePeresPitmanYor2001} proved that, in dimensions $d\geq3$, the occupation
measure of the radial projection of a Brownian path determines its range and terminal point.
Croydon~\cite{Croydon2009} proved that, in dimensions $d\geq8$, the embedded range of a Brownian
spatial tree determines the spatial tree. Duquesne~\cite{Duquesne2009} identified the occupation
measure of super-Brownian motion in dimensions $d\geq5$ with an exact packing measure carried by
its range. 

Our idea of using the decay of measure of close points in the correlation integral is very similar
in spirit to using the Minkowski content to study the size of a set. Write $\cL^d$ for the
$d$-dimensional Lebesgue measure and $\langle F\rangle_r$ for the open $r$-neighbourhood of
$F\subset\bbR^d$.
The Minkowski content, defined
by 
\[
  \cM^s(F) = \lim_{r\downarrow 0} r^{s-d}\cL^d(\langle F\rangle_r),
\]
whenever the limit exists, is particularly useful in the study of random objects that involve a Wiener process. 
Spitzer~\cite{Spitzer1964} and Le Gall~\cite{LeGall1988} show that a Brownian path in $\bbR^d$, with $d\geq3$, has positive finite
two-dimensional Minkowski content almost surely. In the plane, the ordinary two-dimensional
content is zero. However, normalising by the second order term gives convergence.
For standard planar Brownian motion run for a fixed time $T>0$, the
Wiener-sausage estimates of Le Gall~\cite[Theorem~2.1]{LeGall1988} give
\[
  \lim_{r\downarrow0}\log(1/r)\,\cL^2\bigl(\langle\BM([0,T])\rangle_r\bigr)=\pi T
  \qquad\text{almost surely}.
\]
Thus the elapsed time can be recovered from the compact range alone. In particular, the laws of
$\BM([0,T_1])$ and $\BM([0,T_2])$ are mutually singular whenever $T_1,T_2>0$ and $T_1\neq T_2$.

Our approach substitutes the Minkowski content for the correlation integral, but also
requires the set to be `subcritical' in the sense that it is of dimension $s<1$. For such sets,
their first order approximation tends to converge for stochastically self-similar sets.
Gatzouras~\cite{Gatzouras2000}, for instance, obtained almost-sure limits for non-lattice
stochastically self-similar sets and logarithmically averaged limits without the non-lattice
restriction. For deterministic self-similar subsets of the line satisfying the open set
condition, Kombrink and Winter~\cite{KombrinkWinter2020} proved that Minkowski measurability is
equivalent to non-lattice type. For random homogeneous self-similar sets satisfying the uniform
strong open set condition, Z\"ahle~\cite{Zahle2020} proved that the mean Minkowski content is
positive and finite in the non-lattice case and that the mean average Minkowski content is
positive and finite in general. These contents are normalised by the mean Minkowski dimension,
which can exceed the
almost-sure dimension. In contrast, for non-almost-deterministic IFSs satisfying uniform
cylinder separation and bounded branching, Troscheit~\cite{Troscheit2023} showed that the upper
Minkowski content and upper average Minkowski content are infinite almost surely. In the
equicontractive case, the lower Minkowski content is zero. Thus limits of expected neighbourhood
volumes need not describe individual realisations.
We use this idea of being able to distinguish sets by the existence of their Minkowski content, as
well as the regularity of Wiener processes, to distinguish their images.

More is known on the Minkowski content for other `fractal sets', where it has come in useful as a
natural parametrisation, such as for Schramm--Loewner evolution.
Lawler--Rezaei~\cite{LawlerRezaei2015} proved that the
$(1+\kappa/8)$-dimensional Minkowski content of an $\mathrm{SLE}_\kappa$ trace exists almost surely
for $0<\kappa<8$ and agrees, up to normalisation, with its natural parametrisation. This measure is
also the continuum counterpart of renormalised discrete length.
Lawler--Viklund~\cite{LawlerViklund2021} proved convergence of planar loop-erased random walk with
its renormalised length parametrisation to $\mathrm{SLE}_2$ parametrised by its $5/4$-dimensional
Minkowski content. Holden--Li--Sun~\cite{HoldenLiSun2022} proved the corresponding convergence of
the critical percolation interface to $\mathrm{SLE}_6$ parametrised by its $7/4$-dimensional
Minkowski content. These results are close in spirit to \cref{thm:minkowski-reconstruction}, where
small neighbourhoods of an unparametrised random range recover a canonical measure carried by that
range.

In dynamical systems there also exist related inverse problems.
One example is Takens' embedding
theorem~\cite{Takens1981}, which allows the reconstruction of dynamical states from finitely many
successive observations. We refer to Bara\'nski, Gutman and
\'Spiewak~\cite{BaranskiGutmanSpiewak2026} and references therein for more background on Takens' theorem.
One specific setting where the correlation integral arises is 
Caby et al.~\cite[Proposition~1 and Section~5]{CabyFarandaManticaVaientiYiou2019}, where they relate
it to hitting-time statistics and extreme close encounters under suitable
dynamical hypotheses. These `close encounters' are captured by the decay of measures of
neighbourhoods with respect to itself.

\subsection{Proof strategy and organisation}
We begin in \cref{sec:setup} by expressing the expected Brownian correlation integral in terms
of the correlation integral of the time measure. For natural self-similar measures, the latter
satisfies a renewal equation, which we study in \cref{sec:renewal}. After normalisation, it
converges to a constant in the non-lattice case and approaches a periodic function in the lattice
case. In
\cref{sec:smoothing}, we show that the Gaussian transform has non-zero Fourier multipliers and
therefore preserves non-constancy, giving the Brownian profile asymptotics in
\cref{thm:profile-asymptotics}.

To make profile separation detectable from a single image measure, we prove in
\cref{sec:variance} a four-point estimate valid for every $s$-Frostman measure. It gives
\begin{equation}\nonumber
  \Var\bigl(\cC_r(\BM_*\mu)\bigr)\leq C
  \begin{cases}
    r^{6s},&0<s<1/2,\\
    r^3(1+\log(1/r)),&s=1/2,\\
    r^{2s+2},&1/2<s<1.
  \end{cases}
\end{equation}
Each bound is $o(r^{4s})$. These bounds get applied in \cref{sec:concentration}, where a
Borel--Cantelli argument, together with monotonicity,
they give uniform almost-sure concentration of $r^{-2s}\cC_r$ about its expectation.
From this we then deduce \cref{thm:main,thm:cantor-application}.

In \cref{sec:reconstruction}, we recover the natural occupation measure from the compact image
using a separate renewal equation for planar neighbourhood area. 
\Cref{sec:examples} illustrates both the scope and the limitations of profile separation. We
compare two base-five missing-digit sets with distinct profiles, then construct non-isometric
attractors whose natural measures have identical correlation integrals.

Finally, \cref{sec:fixed-dimension} treats the two-contraction family of fixed dimension. We
compute the mean correlation profile in \cref{sec:fixed-dimension-family} and prove in
\cref{sec:fixed-dimension-separation} that it is real analytic and tends to zero as $p\downarrow0$.
Together with positivity, these properties give finite level sets and the separation result
\cref{thm:fixed-dimension-separation}. In \cref{sec:small-dimension}, we prove
strict increase throughout the parameter range for all sufficiently small dimensions, as well as
finiteness of the critical set for every $s\leq1/2$.

\subsection{Back matter}

\paragraph{Lean formalisation.}
The paper is accompanied by a formalisation in the Lean~4 theorem prover over Mathlib.
The formalisation is available at
\url{https://github.com/sascha2718/lean-brownian-images} and on
Palomar~\cite{Troscheit2026Palomar}, where it is also (machine) audited.
The Comparator audit
(\url{https://github.com/leanprover/comparator}) checks the nine numbered results of this section
in \texttt{Challenge.lean} against their proved counterparts in \texttt{Solution.lean}.
The wider development contains 136 proved statements in all and uses no unproved literature
hypotheses. We refer the reader to \texttt{paper-correspondence.yaml} for the correspondence with
the article and to \texttt{docs/correspondence.md} for the differences in the proofs. In
particular, Lean obtains the four-point bound by dyadic summation and uses a weaker numerical
bound in the small-dimension argument that gives the same conclusion.

\paragraph{Acknowledgements.}
We thank Mike Hochman for bringing this question to our attention and Mike Hochman, Adam \'Spiewak, 
Jeffrey Steif and Ville Suomala for helpful discussions.

\paragraph{Funding Information}
This research was partly financed by the Royal Swedish Academy of Sciences
(Kungl.\ Vetenskapsakademien) research stipend MA2025-0070 and the Lisa \& Carl-Gustav Esseens Fund
for Mathematics (2025--26).

\paragraph{AI statement.}
In accordance with the Leiden Declaration on Artificial Intelligence and Mathematics
(\url{https://leidendeclaration.ai}), we disclose the use of large-language-model tools.
LLM models were used heavily to draft much of the Lean code, as well as for light copy-editing the exposition.
Part of the calculations in \cref{sec:fixed-dimension} for the two contraction family of self-similar
sets were first done by LLMs and later verified using symbolic calculation software. 
The mathematical arguments and ideas themselves were human-created, 
with only minor aspects being influenced by the automated formalisation of the results.
All statements and output were revised and checked by the author, who
takes full responsibility for the accuracy of mathematical claims, citations, formal statements, and
prose.

%%%%%
%%%%%
%%%%%
\section{Frostman measures and Gaussian reduction}\label{sec:setup}

Let $0<s<1$. A probability measure $\mu$ on $[0,1]$ is called \emph{$s$-Frostman} if there is a
constant $A\geq1$ such that
\begin{equation}\label{eq:frostman}
  \mu(B(x,\rho))\leq A\rho^s
\end{equation}
for every $x\in[0,1]$ and $0<\rho\leq1$.

By the mass distribution principle, the support of an $s$-Frostman measure has Hausdorff dimension
at least $s$.
Of particular note are Ahlfors--David $s$-regular sets $F$. The $s$-Hausdorff measure $\mu$
restricted to such a set satisfies the stronger condition that $\mu(B(x,r)\cap F)\asymp r^s$
for all $x\in F$ and $0<r<\diam(F)$. Examples are self-conformal (and in particular self-similar)
sets with positive Hausdorff measure at their dimension,
see~\cite[Theorem~3.1]{AngelevskaKaenmakiTroscheit2020} and \cite{FraserHendersonOlsonRobinson2015}.

Fix an $s$-Frostman probability measure $\mu$ and a constant $A$ satisfying \cref{eq:frostman}.
We use the correlation integral notation from \cref{eq:correlation-integral} also for measures
on $[0,1]$. Thus, by Fubini's theorem and the Frostman bound,
\begin{equation}\label{eq:correlation-frostman}
  \cC_\delta(\mu)=\int_{[0,1]}\mu(B(x,\delta))\dd\mu(x)\leq A\delta^s,
  \qquad 0<\delta\leq1.
\end{equation}
The measure $\mu$ has no atoms, so another application of Fubini's theorem gives
\[
  (\mu\times\mu)\set{(x,y)\colon\abs{x-y}=\delta}
  =\int\mu\set{x-\delta,x+\delta}\dd\mu(x)=0,
  \qquad \delta\geq0.
\]
Consequently, $\delta\mapsto\cC_\delta(\mu)$ is continuous and is also the distribution
function of $\abs{X-Y}$ for independent $X,Y$ with law $\mu$. In particular, strict and
non-strict distance inequalities give the same correlation integral, with
$\cC_0(\mu)=0$ and $\cC_\delta(\mu)=1$ for $\delta\geq1$.

The law of a Brownian increment depends only on the time difference. We can therefore compute
the expected correlation integral of $\BM_*\mu$ by integrating over these time differences.

\begin{lemma}\label{thm:gaussian-reduction}
For every $r>0$,
\begin{equation}\label{eq:gaussian-reduction}
  \begin{split}
    \bbE\cC_r(\BM_*\mu)
    &=\int_{[0,1]^2}\left(1-\exp\left(-\frac{r^2}{2\abs{u-t}}\right)\right)
      \dd\mu(u)\dd\mu(t)\\
    &=\int_0^1\left(1-\exp\left(-\frac{r^2}{2\delta}\right)\right)
      \dd\cC_\delta(\mu).
  \end{split}
\end{equation}
\end{lemma}

\begin{proof}
For $u\neq t$, the increment $\BM_u-\BM_t$ is a centred planar Gaussian with covariance
$\abs{u-t}I_2$. Polar integration gives
\[
  \bbP(\abs{\BM_u-\BM_t}<r)=1-\exp\left(-\frac{r^2}{2\abs{u-t}}\right).
\]
Since $\mu$ is atomless, this holds for $\mu\times\mu$-almost every $(u,t)$.
Integrating with respect to $\mu\times\mu$ and applying Fubini's theorem gives the first
identity, while pushing $\mu\times\mu$ forward by $(u,t)\mapsto\abs{u-t}$ gives the second.
\end{proof}

To compare the correlation integrals before and after applying Brownian motion, we introduce the
\emph{normalised correlation profile} of $\mu$,
\[
  G(t)=e^{st}\cC_{e^{-t}}(\mu),\qquad t\in\bbR.
\]
This is continuous, and \cref{eq:correlation-frostman} gives $G(t)\leq A$ for $t\geq0$, while
$G(t)=e^{st}$ for $t\leq0$. In particular, $G$ is bounded on $\bbR$.
Recall from \cref{eq:expected-profile} that
$H_\mu^s(t)=e^{2st}\bbE\cC_{e^{-t}}(\BM_*\mu)$ is the normalised expected correlation
profile. We extend this definition to all $t\in\bbR$.

Since $\cC_\delta(\mu)=1$ for $\delta\geq1$, we may extend the Stieltjes integral in
\cref{eq:gaussian-reduction} to $[0,\infty)$. The boundary terms in integration by parts
vanish: $\cC_\delta(\mu)\to0$ as $\delta\downarrow0$, and
$1-\exp(-r^2/(2\delta))\to0$ as $\delta\to\infty$. Hence
\[
  \bbE\cC_r(\BM_*\mu)
  =\frac{r^2}{2}\int_0^\infty\delta^{-2}
    \exp\left(-\frac{r^2}{2\delta}\right)\cC_\delta(\mu)\dd\delta.
\]
Substituting $\delta=r^2\eta$ and writing $t=\log(1/r)$ now gives
\begin{equation}\label{eq:h-definition}
  H_\mu^s(t)=\int_0^\infty \phi(\eta)G(2t-\log\eta)\dd\eta,
  \qquad
  \phi(\eta)=\frac12\eta^{s-2}\exp\left(-\frac{1}{2\eta}\right).
\end{equation}
The kernel $\phi$ is integrable, so $H_\mu^s$ is bounded and uniformly continuous on $\bbR$.

\section{Renewal of correlation profiles}\label{sec:renewal}

\subsection{Open set condition estimates}\label{sec:osc-geometry}

We now apply the renewal theorem to the correlation integral of the natural measure of a
self-similar IFS on $[0,1]$ satisfying the open set condition.
As before, we write $S_i$ for the contracting maps,
 $r_i$ for the contraction rations, $K$ for the associated attractor, $\mu$ for its natural measure
 and $s\in(0,1)$ for its Hausdorff dimension.
Write $r_{\min}=\min_i r_i$.
For a finite word $w=i_1\cdots i_k$, put
\[
  S_w=S_{i_1}\circ\cdots\circ S_{i_k},\qquad
  r_w=r_{i_1}\cdots r_{i_k},\qquad
  p_w=r_w^s,\qquad I_w=S_w([0,1]).
\]
For a non-empty word $w$, write $w^-$ for the word obtained by deleting its last symbol.
For $0<u<1$, set $\cW(u)=\set{w\ne\varnothing\colon r_w\leq u<r_{w^-}}$, and set
$\cW(1)=\set{\varnothing}$. For $w\in\cW(u)$, the stopping condition and
iteration of self-similarity give
\begin{equation}\label{eq:stopping-family}
  r_{\min}u<r_w\leq u,\qquad
  \sum_{w\in\cW(u)}p_w=1,\qquad
  \mu=\sum_{w\in\cW(u)}p_w(S_w)_*\mu.
\end{equation}

By the standard bounded overlap estimate under the open set condition (see
Hutchinson~\cite[proof of Theorem~5.3(1)]{Hutchinson1981}), there is an integer $M\geq1$
such that, for every $0<u\leq1$ and $x\in\bbR$, at most $M$ words $w\in\cW(u)$
satisfy $x\in I_w$. The intervals $(I_w)_{w\in\cW(u)}$ can therefore be partitioned
into at most $M$ families with pairwise disjoint interiors.

\subsection{Cross-correlation and decay}
We write
\[
  \cC^\times_\delta(\mu)=\sum_{i\neq j}p_ip_j
    (\mu\times\mu)\set{(x,y)\colon\abs{S_i(x)-S_j(y)}\leq\delta}
\]
for the total cross term contribution from pairs with distinct first symbols. Expanding
$\mu=\sum_i p_i(S_i)_*\mu$ and rescaling gives
\begin{equation}\label{eq:correlation-recursion}
  \cC_\delta(\mu)=\sum_i r_i^{2s}\cC_{\delta/r_i}(\mu)+\cC^\times_\delta(\mu),
  \qquad\delta>0.
\end{equation}
Recall that $G(t)=e^{st}\cC_{e^{-t}}(\mu)$, and put
$\alpha_i=\log(1/r_i)$ for the logarithmic contractions. Multiplication by $e^{st}$ at
$\delta=e^{-t}$ gives
\begin{equation}\label{eq:g-recursion}
  G(t)=\sum_i p_iG(t-\alpha_i)+e^{st}\cC^\times_{e^{-t}}(\mu),
  \qquad t\in\bbR.
\end{equation}
Let $\vartheta=\sum_i p_i\delta_{\alpha_i}$ be the probability distribution of the logarithmic
contraction ratios with mean
$m=\sum_i p_i \alpha_i$. Write $z(t)=e^{st}\cC^\times_{e^{-t}}(\mu)$ for the rescaled cross term.
Then
\cref{eq:g-recursion} is the renewal equation $G=z+\vartheta*G$, with interarrival distribution
$\vartheta$ and inhomogeneous term $z$.
We first establish the behaviour of the $z$ term.

\begin{lemma}\label{thm:renewal-defect-left-tail}
The function $z(t)$ is continuous and for $t\leq0$,
\[
  z(t)=e^{st}\left(1-\sum_i r_i^{2s}\right)>0.
\]
In particular, $z(t)$ decays exponentially as $t\to-\infty$.
\end{lemma}

\begin{proof}
Continuity follows from $z(t)=G(t)-\sum_i p_iG(t-\alpha_i)$, since $G$ is continuous,
and non-negativity follows from the definition of the cross term. For $t\leq0$, every
pair lies within distance $e^{-t}\geq1$, so
\[
  \cC^\times_{e^{-t}}(\mu)=\sum_{i\neq j}p_ip_j=1-\sum_i p_i^2.
\]
The claim follows because $p_i=r_i^s\in(0,1)$ and $\sum_i p_i=1$.
\end{proof}

The inhomogeneous term disappears for large enough $t$ when the set satisfies the strong separation
condition. Assuming just the open set condition requires a little extra work and the following bound
still gives exponential decay of $z(t)$ as $t\to\infty$ here.

\begin{lemma}\label{thm:cross-piece-mass}
There exist $C>0$ and $\eta\in(0,1-s)$ such that
\begin{equation}\label{eq:cross-piece-mass}
  0\leq\cC^\times_\delta(\mu)\leq C\delta^{s+\eta},\qquad0<\delta\leq1.
\end{equation}
Consequently, $0\leq z(t)\leq Ce^{-\eta t}$ for all $t\geq0$.
\end{lemma}

\begin{proof}
  By Ahlfors regularity, for each $x\in K$ the $\mu$-mass of points $y$ satisfying
  $\abs{S_i(x)-S_j(y)}\leq\delta$ is at most $C\delta^s$, and it vanishes unless
  $\dist(S_i(x),S_jK)\leq\delta$. Fubini's theorem gives
  \[
    \cC^\times_\delta(\mu)\leq C\delta^s\sum_{i\neq j}p_ip_j\,
    \mu\set{x\in K\colon\dist(S_i(x),S_jK)\leq\delta},
  \]
  so it suffices to prove, for some $\eta\in(0,1-s)$,
  \begin{equation}\label{eq:boundary-mass}
    \mu\set{x\in K\colon\dist(S_i(x),S_jK)\leq\delta}
    \leq C\delta^\eta,\qquad0<\delta\leq1,\quad i\neq j.
  \end{equation}
  By Schief~\cite{Schief1994} the IFS satisfies the strong open set condition, so
  Lalley's estimate~\cite[Proposition~3]{Lalley1988} applies: there is $\gamma>0$
  such that every $\delta$-separated subset of
  $A_{ij}(\delta)=\set{y\in S_iK\colon\dist(y,S_jK)\leq\delta}$ has at most
  $C\delta^{\gamma-s}$ points. The closed balls of radius $\delta$ about a maximal
  $\delta$-separated subset cover $A_{ij}(\delta)$, and each has $\mu$-mass at most
  $C\delta^s$ by Ahlfors regularity, so $\mu(A_{ij}(\delta))\leq C\delta^\gamma$.
  Since $\mu(S_iA)\geq p_i\mu(A)$ for every Borel set $A$ by self-similarity,
  \cref{eq:boundary-mass} follows with $\eta=\min\set{\gamma,(1-s)/2}$, and with it
  \cref{eq:cross-piece-mass}.
  Taking $\delta=e^{-t}$ and multiplying by $e^{st}$ gives the bound on $z$.
\end{proof}

\subsection{Renewal and asymptotics}\label{sec:renewal-equation}
We use the following form of the key renewal theorem, see
Falconer~\cite[Proposition~7.1 and Theorem~7.2]{Falconer1997} or
Feller~\cite[Chapter~XI]{Feller1971}. Let $\vartheta$ be a finitely supported probability
distribution on $(0,\infty)$ with mean $m$. Suppose that a bounded continuous function
$f\colon\bbR\to\bbR$ satisfies $f=z+\vartheta*f$, tends to zero at $-\infty$, and has a
continuous, possibly signed defect $z$ with $\abs{z(t)}\leq Ce^{-c\abs{t}}$ for some $C,c>0$. Then
\[
  f=\sum_{n\geq0}\vartheta^{*n}*z,
\]
and $f$ is uniformly continuous. If $\vartheta$ is non-lattice, then
\[
  f(t)\to\frac1m\int_\bbR z(x)\dd x
  \qquad\text{as }t\to\infty.
\]
If $\vartheta$ is lattice with span $\ell$, then
\[
  \sup_{t\in[0,\ell]}\abs{f(t+n\ell)-\widetilde f(t)}\to0\quad(n\to\infty),
  \qquad
  \widetilde f(t)=\frac{\ell}{m}\sum_{k\in\bbZ}z(t+k\ell).
\]
Indeed, the renewal theorem gives pointwise convergence on $[0,\ell]$. The exponential tails of
$z$ make $\widetilde f$ continuous, and uniform continuity of $f$ makes its translates
equicontinuous, so the convergence is uniform on this compact interval. In particular,
$f(t)-\widetilde f(t)\to0$ as $t\to\infty$.

We shall also use the following consequence for a bounded continuous function
$f\colon[0,\infty)\to\bbR$. If $\abs{f(t)-(\vartheta*f)(t)}\leq Ce^{-ct}$ for all sufficiently
large $t$, then, for every $u$ in the support of $\vartheta$,
\begin{equation}\label{eq:renewal-shift}
  f(t-u)-f(t)\to0
  \qquad\text{as }t\to\infty.
\end{equation}
To see this, extend $\min\{t,1\}f(t)$ by zero to $t\leq0$. The defect of this cut-off function
is continuous, vanishes on $(-\infty,0]$, and decays exponentially at $\infty$, so the renewal
theorem applies. The constant limit gives \cref{eq:renewal-shift} in the non-lattice case.
In the lattice case, the uniform periodic approximation gives the same conclusion since
$u\in \ell\bbZ$.

We apply this result to the correlation profile $G$. In the lattice case with span $\ell$, define
\begin{equation}\label{eq:lattice-periodisation}
  \widetilde G(t)=\frac{\ell}{m}\sum_{k\in\bbZ}z(t+k\ell),\qquad t\in\bbR,
\end{equation}
which is continuous and $\ell$-periodic by the exponential bounds on $z$.

\begin{proposition}\label{thm:renewal-limits}
  If the IFS is of non-lattice type, then
  \begin{equation}\label{eq:g-non-lattice-limit}
    G(t)\to\frac{1}{m}\int_{\bbR}z(x)\dd x\in(0,\infty),
    \qquad t\to\infty.
  \end{equation}
  If the IFS is of lattice type with span $\ell$, then $\widetilde G$ is continuous,
  strictly positive and $\ell$-periodic, and
  \[
    G(t)-\widetilde G(t)\to0,\qquad t\to\infty.
  \]
\end{proposition}

\begin{proof}
By \cref{thm:renewal-defect-left-tail,thm:cross-piece-mass}, the defect $z$ satisfies the
hypotheses above. Since $G$ is bounded and tends to zero at $-\infty$, the renewal theorem
gives \cref{eq:g-non-lattice-limit} and the lattice approximation. Positivity of the integral
and of $\widetilde G$ follows from $z\geq0$ and $z(t)>0$ for $t\leq0$.
\end{proof}

We next show that in the lattice case, the limiting function is properly $\ell$ periodic.

\begin{lemma}\label{thm:lattice-nonconvergence}
If the IFS is of lattice type, then the periodic profile $\widetilde G$ is non-constant.
\end{lemma}

\begin{proof}
Fix $i\neq j$. The maps $S_{ij}$ and $S_{ji}$ have the contraction ratio, while their
images must be non-trivially translationally separated by the open set condition.
So, there exists $\rho\coloneqq\abs{S_{ij}(0)-S_{ji}(0)}>0$.

For $0<\delta<\rho/(2r_{ij})$, put $\delta_1=\rho-2r_{ij}\delta$ and
$\delta_2=\rho+r_{ij}\delta$. We have
$\delta_1<\abs{S_{ij}(x)-S_{ji}(y)}\leq\delta_2$.
Expanding the difference $\cC^\times_{\delta_2}(\mu)-\cC^\times_{\delta_1}(\mu)$ therefore gives
\begin{align*}
  \cC^\times_{\delta_2}(\mu)-\cC^\times_{\delta_1}(\mu)
  &\geq p_i^2p_j^2(\mu\times\mu)
    \set{(x,y)\colon\delta_1<\abs{S_{ij}(x)-S_{ji}(y)}\leq\delta_2}\\*
  &\geq p_i^2p_j^2\cC_\delta(\mu).
\end{align*}
Hence,
\begin{align*}
  \delta_2^s\widetilde G(-\log\delta_2)-\delta_1^s\widetilde G(-\log\delta_1)
  &=\frac{\ell}{m}\sum_{k\in\bbZ}e^{sk\ell}
    \bigl(\cC^\times_{\delta_2e^{-k\ell}}(\mu)-\cC^\times_{\delta_1e^{-k\ell}}(\mu)\bigr)\\*
  &\geq\frac{\ell}{m} p_i^2p_j^2\cC_\delta(\mu)\geq c\delta^s.
\end{align*}
Since $\widetilde G$ is bounded and $\delta_2^s-\delta_1^s=O(\delta)$, this gives
\[
  \delta_2^s\bigl(\widetilde G(-\log\delta_2)-\widetilde G(-\log\delta_1)\bigr)
  \geq c\delta^s-C\delta>0
\]
for all sufficiently small $\delta>0$, since $s<1$.
\end{proof}

\section{Gaussian smoothing and distributional separation}\label{sec:smoothing}

For $0<s<1$, $\ell>0$, and a continuous $\ell$-periodic function $g$, define $Tg$ by the
Gaussian smoothing from \cref{eq:h-definition},
\[
  Tg(t)=\int_0^\infty\phi(\eta)g(2t-\log\eta)\dd\eta.
\]
Then $Tg$ is continuous and $\ell/2$-periodic.

\begin{lemma}\label{thm:smoothing-injective}
The operator $T$ is injective on the continuous $\ell$-periodic functions. Further,
$Tg$ can be constant only when $g$ is constant.
\end{lemma}

\begin{proof}
Since $g$ is $\ell$-periodic, its Fourier coefficients are
\[
  \widehat g(k)=\frac1\ell\int_0^\ell g(x)e^{-2\pi ikx/\ell}\dd x,\qquad k\in\bbZ.
\]
The $\ell$-periodic function $x\mapsto Tg(x/2)=\int_0^\infty\phi(\eta)g(x-\log\eta)\dd\eta$
is an average of translates of $g$. Translating $g$ by $\log\eta\in\bbR$ multiplies
$\widehat g(k)$ by $e^{-2\pi ik\log(\eta)/\ell}=\eta^{-2\pi ik/\ell}$. 
Averaging over $\eta$, Fubini's theorem gives
\begin{equation}\label{eq:fourier-multiplier}
  \widehat{Tg}(k)=\left(\int_0^\infty\phi(\eta)\eta^{-2\pi ik/\ell}\dd\eta\right)\widehat g(k),
  \qquad k\in\bbZ,
\end{equation}
where $\widehat{Tg}(k)$ denotes the $k$-th Fourier coefficient of $Tg$ with respect to
its period $\ell/2$.

The multiplier, when substituting $x=1/(2\eta)$, becomes
\[
  \int_0^\infty\phi(\eta)\eta^{-2\pi ik/\ell}\dd\eta
  =2^{-s+2\pi ik/\ell}\int_0^\infty x^{-s+2\pi ik/\ell}e^{-x}\dd x.
\]
The last integral is $\Gamma(z)=\int_0^\infty x^{z-1}e^{-x}\dd x$ for
$z=1-s+2\pi ik/\ell$. Thus
\begin{equation}\label{eq:gamma-multiplier}
  \int_0^\infty\phi(\eta)\eta^{-2\pi ik/\ell}\dd\eta
  =2^{-s+2\pi ik/\ell}\Gamma(1-s+2\pi ik/\ell),
\end{equation}
which is non-zero.

Now let $g_1$ and $g_2$ be continuous and $\ell$-periodic with $Tg_1=Tg_2$. Linearity
gives $T(g_1-g_2)=0$, so by \cref{eq:fourier-multiplier} and the non-vanishing of the
multipliers, $\widehat{(g_1-g_2)}(k)=0$ for every $k\in\bbZ$. By the uniqueness theorem
$g_1=g_2$, so $T$ is injective. The conclusion on constant functions follows directly from the
non-vanishing of multipliers.
\end{proof}

\begin{proof}[Proof of \cref{thm:profile-asymptotics}]
In the non-lattice case, \cref{thm:renewal-limits} and dominated convergence in
\cref{eq:h-definition} give
\[
  H_\mu^s(t)\to\frac1m\int_\bbR z(x)\dd x
    \int_0^\infty\phi(\eta)\dd\eta\in(0,\infty).
\]

In the lattice case with span $\ell$, let $\widetilde G$ be the periodic limit from
\cref{thm:renewal-limits} and put $\widetilde H_\mu^s=T\widetilde G$. This function is
continuous, strictly positive, and $\ell/2$-periodic, and it is non-constant by
\cref{thm:lattice-nonconvergence,thm:smoothing-injective}. Since $G$ and $\widetilde G$ are
bounded, dominated convergence gives
\[
  H_\mu^s(t)-\widetilde H_\mu^s(t)
  =\int_0^\infty\phi(\eta)
    \bigl(G(2t-\log\eta)-\widetilde G(2t-\log\eta)\bigr)\dd\eta
  \to0.\qedhere
\]
\end{proof}

\section{The four-point variance bound}\label{sec:variance}

In this section we establish a variance bound on $\cC_r(\BM_*\mu)$ by first considering Gaussian
vectors coming from the Brownian path at different points in time.
Fix $0<s<1$ and an $s$-Frostman probability measure $\mu$ on $[0,1]$ with constant
$A\geq1$. We use the variance scaling function
\begin{equation}\label{eq:variance-scale}
  V_s(r)\coloneqq
  \begin{cases}
    r^{6s},&0<s<1/2,\\
    r^3(1+\log(1/r)),&s=1/2,\\
    r^{2s+2},&1/2<s<1.
  \end{cases}
\end{equation}
For $(u,t,u',t')\in[0,1]^4$, we write $I$ and $I'$ for the open intervals with endpoints
$u,t$ and $u',t'$, respectively, and put $Z=\BM_u-\BM_t$ and $Z'=\BM_{u'}-\BM_{t'}$ for the
associated random vectors of differences in the Brownian motion.

\begin{proposition}\label{thm:four-point-integral}
There is a constant $C>0$, depending only on $s$ and $A$, such that, for every $0<r\leq1$,
\[
  \int_{\set{I\cap I'\neq\varnothing}}
  \bbP(\abs{Z}<r,\abs{Z'}<r)\dd\mu^4
  \leq C V_s(r).
\]
\end{proposition}

\begin{proof}
Since $\mu$ is non-atomic, we may ignore tuples with repeated endpoints. For overlapping
intervals, put
\[
  a=|I\cap I'|\quad\text{and}\quad b=|I\mathbin{\triangle}I'|.
\]
The coordinate pairs $(Z_1,Z'_1)$ and $(Z_2,Z'_2)$ are independent centred Gaussian vectors,
each with covariance matrix
\[
  \begin{pmatrix}|I|&\pm a\\\pm a&|I'|\end{pmatrix}.
\]
The determinant of this matrix is
\[
  |I||I'|-a^2=ab+(|I|-a)(|I'|-a)\geq ab>0.
\]
Hence the joint density $f$ of $(Z,Z')$ satisfies
\[
  f(z,z')\leq\frac{1}{(2\pi)^2(|I||I'|-a^2)}\leq\frac{1}{(2\pi)^2ab}.
\]
Hence
\[
  \bbP(\abs{Z}<r,\abs{Z'}<r)
  =\int_{B(0,r)}\int_{B(0,r)}f(z,z')\dd z'\dd z
  \leq\frac{(\pi r^2)^2}{(2\pi)^2ab}
  =\frac{r^4}{4ab}.
\]
We may assume, relabelling if necessary, that $|I|\geq|I'|$, so $|I|\geq(a+b)/2$. Since
\[
  \{\abs{Z}<r,\abs{Z'}<r\}\subseteq\{\abs{Z}<r\},
\]
we obtain
\[
  \bbP(\abs{Z}<r,\abs{Z'}<r)
  \leq\int_{B(0,r)}\frac{e^{-\abs{z}^2/(2|I|)}}{2\pi|I|}\dd z
  =1-e^{-r^2/(2|I|)}
  \leq\frac{r^2}{2|I|}\leq\frac{r^2}{a+b}.
\]
Together with the trivial full probability bound, these inequalities yield
\begin{align*}
  \bbP(\abs{Z}<r,\abs{Z'}<r)
  &\leq\min\left\{1,\frac{r^2}{a+b},\frac{r^4}{4ab}\right\}\\
  &\leq\min\left\{1,\frac{r^2}{a},\frac{r^2}{b},\frac{r^4}{ab}\right\}
  =\min\left\{1,\frac{r^2}{a}\right\}\min\left\{1,\frac{r^2}{b}\right\}.
\end{align*}

The Frostman bound and $\mu(\bbR)=1$ give, for $a,b>0$,
\[
  \mu^4\set{0<|I\cap I'|<a,\ |I\mathbin{\triangle}I'|<b}
  \leq4!A^3\min\{1,a\}^s\min\{1,b\}^{2s}.
\]
These bounds and Fubini's theorem give
\begin{align*}
  \int_{\set{I\cap I'\neq\varnothing}}
    \bbP(\abs{Z}<r,\abs{Z'}<r)\dd\mu^4
  &\leq r^4\int_{r^2}^\infty\int_{r^2}^\infty
    \mu^4\set{0<|I\cap I'|<a,\ |I\mathbin{\triangle}I'|<b}
    \frac{\dd b\dd a}{a^2b^2}\\
  &\leq4!A^3
    \left(r^2\int_{r^2}^\infty\frac{\min\{1,a\}^s}{a^2}\dd a\right)
    \left(r^2\int_{r^2}^\infty\frac{\min\{1,b\}^{2s}}{b^2}\dd b\right).
\end{align*}
Their product is at most $C V_s(r)$, since
\begin{align*}
  r^2\int_{r^2}^\infty\frac{\min\{1,a\}^s}{a^2}\dd a
  &=r^2+r^2\int_{r^2}^1a^{s-2}\dd a
  =\frac{r^{2s}-sr^2}{1-s}\leq\frac{r^{2s}}{1-s},
\end{align*}
and
\begin{align*}
  r^2\int_{r^2}^\infty\frac{\min\{1,b\}^{2s}}{b^2}\dd b
  &=r^2+r^2\int_{r^2}^1b^{2s-2}\dd b\\*
  &\leq
  \begin{cases}
    r^{4s}/(1-2s),&0<s<1/2,\\
    r^2(1+2\log(1/r)),&s=1/2,\\
    2sr^2/(2s-1),&1/2<s<1.
  \end{cases}
\end{align*}
These are the required bounds, up to a multiplicative factor.
\end{proof}

\begin{theorem}\label{thm:variance}
Let $0<s<1$ and let $\mu$ be an $s$-Frostman probability measure on $[0,1]$ with Frostman
constant $A\geq1$. There is a constant $C>0$, depending only on $s$ and $A$, such that
\[
  \Var\bigl(\cC_r(\BM_*\mu)\bigr)\leq C V_s(r),
  \qquad 0<r\leq1.
\]
Moreover, $V_s(r)=o(r^{4s})$ as $r\to0$.
\end{theorem}

\begin{proof}
Expanding the variance and applying Fubini's theorem gives
\[
  \Var\bigl(\cC_r(\BM_*\mu)\bigr)=\int\left(
  \bbP(\abs{Z}<r,\abs{Z'}<r)
  -\bbP(\abs{Z}<r)\bbP(\abs{Z'}<r)
  \right)\dd\mu^4.
\]
By independence of disjoint Brownian increments, the integrand vanishes when
$I\cap I'=\varnothing$. Otherwise it is at most $\bbP(\abs{Z}<r,\abs{Z'}<r)$, so
\cref{thm:four-point-integral} gives the variance bound. In each case of
\cref{eq:variance-scale}, $V_s(r)/r^{4s}\to0$ as $r\to0$.
\end{proof}

\section{Uniform concentration and mutual singularity}\label{sec:concentration}

Let $s$, $\mu$ and $A$ be as in \cref{sec:variance}. For $t\geq0$, put
\[
  Y_\mu(t)\coloneqq e^{2st}\cC_{e^{-t}}(\BM_*\mu).
\]
Then $\bbE Y_\mu(t)=H_\mu^s(t)$ by \cref{eq:expected-profile}. Moreover, \cref{thm:variance} gives a
constant $C>0$, depending only on $s$ and $A$, such that, for every $t\geq0$,
\begin{equation}\label{eq:y-variance}
  \Var(Y_\mu(t))\leq C
  \begin{cases}
    e^{-2st},&0<s<1/2,\\
    (1+t)e^{-t},&s=1/2,\\
    e^{-2(1-s)t},&1/2<s<1.
  \end{cases}
\end{equation}
This quickly decaying variance bound, gives almost sure convergence by a standard argument:

\begin{theorem}\label{thm:uniform-concentration}
Let $0<s<1$ and let $\mu$ be an $s$-Frostman probability measure on $[0,1]$. Then
\[
  \limsup_{t\to\infty}\abs{Y_\mu(t)-H_\mu^s(t)}=0
  \qquad\text{almost surely}.
\]
\end{theorem}

\begin{proof}
Fix an integer $m\geq1$ and put $t_{j,m}\coloneqq j/m$. By \cref{eq:y-variance},
\[
  \sum_{j=0}^\infty\Var(Y_\mu(t_{j,m}))<\infty.
\]
One can then estimate $\bbP\{\abs{Y_\mu(t)-H_\mu^s(t)}>\eps\}$ by Chebyshev's inequality and apply
the Borel--Cantelli lemma, giving
\begin{equation}\label{eq:grid-convergence}
  Y_\mu(t_{j,m})-H_\mu^s(t_{j,m})\to0
  \qquad\text{almost surely as }j\to\infty.
\end{equation}

If $t\in[t_{j,m},t_{j+1,m}]$, monotonicity of $r\mapsto\cC_r(\BM_*\mu)$ gives
\begin{equation}\label{eq:monotone-fill}
  e^{-2s/m}Y_\mu(t_{j+1,m})\leq Y_\mu(t)\leq e^{2s/m}Y_\mu(t_{j,m}).
\end{equation}
Combining \cref{eq:grid-convergence,eq:monotone-fill} gives, almost surely,
\[
  \limsup_{t\to\infty}\abs{Y_\mu(t)-H_\mu^s(t)}
  \leq (e^{2s/m}-1)\sup_{t\geq0}H_\mu^s(t)
  +\sup_{\substack{t,u\geq0\\\abs{t-u}\leq1/m}}
  \abs{H_\mu^s(t)-H_\mu^s(u)}.
\]
Taking the countable intersection over $m$ and using the boundedness of $H_\mu^s$
and its uniform continuity from \cref{eq:h-definition} gives the desired limit.
\end{proof}

\begin{proof}[Proof of \cref{thm:main}]
For $i\in\set{1,2}$, define
\[
  \cE_i\coloneqq
  \set{\nu\in\cP(\bbR^2)\colon
  \lim_{n\to\infty}\sup_{\substack{q\in\bbQ\\q\geq n}}
  \abs{e^{2sq}\cC_{e^{-q}}(\nu)-H_{\mu_i}^s(q)}=0}.
\]
Each $\cE_i$ is Borel, since $\nu\mapsto\cC_r(\nu)$ is Borel for every $r>0$.
Suppose that $\nu\in\cE_1\cap\cE_2$. The triangle inequality gives
\[
  \lim_{n\to\infty}\sup_{\substack{q\in\bbQ\\q\geq n}}
  \abs{H_{\mu_1}^s(q)-H_{\mu_2}^s(q)}=0.
\]
Both expected profiles are uniformly continuous by \cref{eq:h-definition}.
Density of the rationals therefore implies $H_{\mu_1}^s(t)-H_{\mu_2}^s(t)\to0$ as
$t\to\infty$, contradicting the assumption of distinct profiles, \cref{eq:profile-separation}. Hence
$\cE_1\cap\cE_2=\varnothing$. 

Further, by \cref{thm:uniform-concentration},
\[
  \Law(\BM_*\mu_i)(\cE_i)=1.
\]
Hence the two laws give
full measure to the respective disjoint Borel sets $\cE_1$ and $\cE_2$, which proves
that they are mutually singular.
\end{proof}

We now distinguish lattice and non-lattice measures.
\begin{proof}[Proof of \cref{thm:cantor-application}]
The IFS defining $\muA$ is of lattice type, so \cref{thm:profile-asymptotics}
gives a non-constant periodic asymptotic profile for $H_{\muA}^s$, while $H_{\muB}^s$ converges
to a constant. Thus \cref{eq:profile-separation} holds, and \cref{thm:main} gives mutual
singularity.
\end{proof}

\section{Minkowski reconstruction from the image set}\label{sec:reconstruction}
In this section we prove \cref{thm:minkowski-reconstruction}. 
Fix a self-similar IFS, and use the notation of \cref{sec:renewal}.

\subsection{Neighbourhood estimates}

Write $E=\BM(K)$.
For $x,y\in\bbR^2$, a non-empty compact set $F\subset\bbR^2$, and $r,q>0$, the cut-off functions
and their integrals satisfy
\begin{equation}\label{eq:neighbourhood-union-scaling}
  \begin{split}
    \left(1-\frac{\dist(x,E)}r\right)_+
    &=\max_{1\leq i\leq N}\left(1-\frac{\dist(x,\BM(S_i(K)))}r\right)_+,\\
    M_r(y+qF)
    &=\int_{\bbR^2}\left(1-\frac{\dist(x,y+qF)}r\right)_+\dd x
      =q^2M_{r/q}(F).
  \end{split}
\end{equation}
We define the \emph{overlap defect} by
\[
  D_r\coloneqq\sum_{i=1}^NM_r(\BM(S_i(K)))-M_r(E).
\]
Fix $x\in\bbR^2$. If $x$ belongs to exactly $k\geq1$ of the neighbourhoods
$\langle\BM(S_i(K))\rangle_r$, exactly $k$ of the cut-off functions are non-zero at $x$.
There are $\binom{k}{2}$ pairs of neighbourhoods containing $x$. Thus the sum of the
indicators of their pairwise intersections equals $\binom{k}{2}\geq k-1$ at $x$.
If $x$ belongs to no neighbourhood, every cut-off function and every such indicator is zero
at $x$. By the first identity in \cref{eq:neighbourhood-union-scaling}, we therefore have
\[
  \begin{split}
    0&\leq\sum_{i=1}^N\left(1-\frac{\dist(x,\BM(S_i(K)))}r\right)_+
      -\left(1-\frac{\dist(x,E)}r\right)_+\\
    &\leq\sum_{1\leq i<j\leq N}
      \mathbf{1}_{\langle\BM(S_i(K))\rangle_r\cap\langle\BM(S_j(K))\rangle_r}(x).
  \end{split}
\]
Integrating this over $\bbR^2$ gives
\begin{equation}\label{eq:neighbourhood-defect-elementary}
  \begin{split}
    0\leq D_r
    &=\int_{\bbR^2}\left[
      \sum_{i=1}^N\left(1-\frac{\dist(x,\BM(S_i(K)))}r\right)_+
      -\left(1-\frac{\dist(x,E)}r\right)_+\right]\dd x\\
    &\leq\sum_{1\leq i<j\leq N}\int_{\bbR^2}
      \mathbf{1}_{\langle\BM(S_i(K))\rangle_r\cap\langle\BM(S_j(K))\rangle_r}(x)\dd x\\
    &=\sum_{1\leq i<j\leq N}
      \cL^2(\langle\BM(S_i(K))\rangle_r\cap\langle\BM(S_j(K))\rangle_r).
  \end{split}
\end{equation}

We next bound the moments of the neighbourhood area and the cut-off mass.

\begin{lemma}\label{thm:neighbourhood-moments}
For every $q\geq1$, there is $C_q>0$ such that, for $0<r\leq1$,
\begin{equation}\label{eq:neighbourhood-moment-bound}
  \bbE\bigl[\cL^2(\langle E\rangle_r)^q\bigr]
  +\bbE\bigl[M_r(E)^q\bigr]
  \leq C_qr^{q(2-2s)}.
\end{equation}
There is also $c>0$ such that, for $0<r\leq1$,
\begin{equation}\label{eq:neighbourhood-mean-lower-bound}
  \bbE M_r(E)\geq cr^{2-2s}.
\end{equation}
\end{lemma}

\begin{proof}
We apply the stopping construction of \cref{eq:stopping-family} at scale $r^2$, giving
\begin{equation}\label{eq:neighbourhood-stopping-count}
  \#\cW(r^2)\leq Cr^{-2s}.
\end{equation}
For each $w\in\cW(r^2)$, choose an endpoint $t_w$ of $S_w([0,1])$ and put
$\omega_w\coloneqq\sup_{t\in S_w([0,1])}\abs{\BM_t-\BM_{t_w}}$. The resulting discs
$\overline B(\BM_{t_w},\omega_w)$ cover $E$. Enlarging their radii by $r$ gives
\begin{equation}\label{eq:neighbourhood-cover-bound}
  M_r(E)\leq\cL^2(\langle E\rangle_r)
  \leq\pi\sum_{w\in\cW(r^2)}(r+\omega_w)^2.
\end{equation}
Brownian scaling and the maximal inequality give, for every $w\in\cW(r^2)$,
\[
  \bbE(r+\omega_w)^{2q}\leq C_qr^{2q}.
\]
By convexity,
\[
  \left(\sum_w(r+\omega_w)^2\right)^q
  \leq\bigl(\#\cW(r^2)\bigr)^{q-1}\sum_w(r+\omega_w)^{2q}.
\]
Taking expectations and using \cref{eq:neighbourhood-stopping-count,eq:neighbourhood-cover-bound}
proves
\cref{eq:neighbourhood-moment-bound}.

For the lower bound, let $\nu\coloneqq\BM_*\mu$ and define
$h_{r/2}\colon\bbR^2\to[0,1]$ by $h_{r/2}(x)\coloneqq\nu(B(x,r/2))$. Since $h_{r/2}$ vanishes
outside $\langle E\rangle_{r/2}$, Fubini's theorem gives
\begin{align*}
  \int_{\bbR^2}h_{r/2}(x)\dd x&=\frac{\pi r^2}{4},\\
  \int_{\bbR^2}h_{r/2}(x)^2\dd x
  &=\iint\cL^2(B(y,r/2)\cap B(z,r/2))\dd\nu(y)\dd\nu(z)\\
  &\leq\frac{\pi r^2}{4}\cC_r(\nu).
\end{align*}
Some ball of radius $r/2$ in a finite cover of the support of $\nu$ has positive $\nu$-measure.
Since $\cC_r(\nu)\geq\nu(B(x,r/2))^2$ for every $x\in\bbR^2$, we have $\cC_r(\nu)>0$.
Cauchy--Schwarz therefore gives
\begin{equation}\label{eq:neighbourhood-correlation-lower-bound}
  \cL^2(\langle E\rangle_{r/2})\geq\frac{\pi r^2}{4\cC_r(\nu)}.
\end{equation}
Since $M_r(E)\geq\frac12\cL^2(\langle E\rangle_{r/2})$, Jensen's inequality and
\cref{eq:neighbourhood-correlation-lower-bound} give
\[
  \bbE M_r(E)
  \geq\frac{\pi r^2}{8}\bbE\left[\frac{1}{\cC_r(\nu)}\right]
  \geq\frac{\pi r^2}{8\bbE \cC_r(\nu)}.
\]
For $0<r\leq1$, \cref{eq:expected-profile} and the boundedness of $H_\mu^s$ give
$\bbE\cC_r(\nu)\leq Cr^{2s}$ and hence \cref{eq:neighbourhood-mean-lower-bound}.
\end{proof}

We next show that the expected overlap defect is of smaller order than the expected
neighbourhood mass.
Let $\eta\in(0,1-s)$ be the exponent from \cref{thm:cross-piece-mass}.

\begin{lemma}\label{thm:neighbourhood-overlap}
There exists $C>0$ such that, for all distinct $i,j\in\{1,\ldots,N\}$ and all $0<r\leq1$,
\[
  \bbE\cL^2(\langle\BM(S_i(K))\rangle_r\cap\langle\BM(S_j(K))\rangle_r)
  \leq Cr^{2-2s+2\eta}.
\]
Moreover,
\begin{equation}\label{eq:neighbourhood-defect-mean}
  \bbE D_r\leq Cr^{2-2s+2\eta}.
\end{equation}
\end{lemma}

\begin{proof}
Fix $i\neq j$. The intersection is contained in $\langle E\rangle_r$, so for $r^2\geq r_{\min}$,
\cref{thm:neighbourhood-moments} bounds its expected area by
$Cr^{2-2s}\leq Cr_{\min}^{-\eta}r^{2-2s+2\eta}$.
We may therefore assume $r^2<r_{\min}$ and consider the stopping set $\cW(r^2)$.
For $w,v\in\cW(r^2)$, put
\[
  A_w=\langle\BM(S_wK)\rangle_r,
  \qquad d_{wv}=\dist(I_w,I_v),
\]
where $I_v = S_v([0,1])$.
Since each interval $I_w$ has length at most $r^2$, the Brownian maximal estimate used in
\cref{eq:neighbourhood-cover-bound} gives $\bbE\cL^2(A_w)\leq Cr^2$.
We claim that
\begin{equation}\label{eq:stopping-brownian-overlap}
  \bbE\cL^2(A_w\cap A_v)\leq\frac{Cr^4}{r^2+d_{wv}}.
\end{equation}
For $d_{wv}\leq r^2$, this follows by bounding the intersection by $A_w$.
For $d_{wv}>r^2$, we use the Gaussian-density argument from the proof of
\cref{thm:four-point-integral} which similarly gives
\begin{align*}
  \bbE\cL^2(A_w\cap A_v)
  &\leq\frac{1}{2\pi d_{wv}}\bbE\bigl[\cL^2(A_w)\cL^2(A_v)\bigr]\\
  &=\frac{1}{2\pi d_{wv}}
       \bbE\cL^2(A_w)\,\bbE\cL^2(A_v)
  \leq\frac{Cr^4}{d_{wv}},
\end{align*}
which proves \cref{eq:stopping-brownian-overlap}.

For $h\geq r^2$, let $N_{ij}(h)$ count pairs $(w,v)\in\cW(r^2)^2$ whose first
symbols are $i,j$, respectively, and for which $d_{wv}\leq h$. Every pair of points in
$S_wK\times S_vK$ then has distance at most $h+2r^2\leq3h$.
The cross correlation bound \cref{eq:cross-piece-mass} extends to all positive scales after increasing $C$,
since $\cC^\times_\delta(\mu)\leq1$.
Decompose with the stopping family
and using $p_wp_v\geq r_{\min}^{2s}r^{4s}$, gives
\begin{equation}\label{eq:stopping-close-pairs}
  \begin{split}
    r_{\min}^{2s}r^{4s}N_{ij}(h)
    &\leq\sum_{\substack{w,v\in\cW(r^2)\\w_1=i,\ v_1=j,\ d_{wv}\leq h}}p_wp_v\\
    &\leq\cC^\times_{3h}(\mu)\leq Ch^{s+\eta}.
  \end{split}
\end{equation}

We group the pairs with $d_{wv}\leq r^2$ together and use the ranges
$2^{k-1}r^2<d_{wv}\leq2^kr^2$ for $k\geq1$. Using subadditivity,
\cref{eq:stopping-brownian-overlap,eq:stopping-close-pairs}, and $s+\eta<1$, we obtain
\begin{align*}
  \bbE\cL^2(\langle\BM(S_iK)\rangle_r\cap\langle\BM(S_jK)\rangle_r)
  &\leq Cr^4\sum_{k=0}^\infty\frac{N_{ij}(2^kr^2)}{2^kr^2}\\
  &\leq Cr^{2-2s+2\eta}\sum_{k=0}^\infty2^{k(s+\eta-1)}
  =\frac{Cr^{2-2s+2\eta}}{1-2^{s+\eta-1}}.
\end{align*}
Lastly, by \cref{eq:neighbourhood-defect-elementary},
\begin{align*}
  \bbE D_r
  &\leq\sum_{1\leq i<j\leq N}
    \bbE\cL^2(\langle\BM(S_iK)\rangle_r\cap\langle\BM(S_jK)\rangle_r)\\
  &\leq\frac{C\binom{N}{2}}{1-2^{s+\eta-1}}r^{2-2s+2\eta},
\end{align*}
which finishes the proof.
\end{proof}

\subsection{Renewal and concentration}

We first show that the expected neighbourhood mass is asymptotically unchanged by the
Brownian scaling shifts.
For $t\geq0$, put
\[
  Q(t)\coloneqq e^{(2-2s)t}\bbE M_{e^{-t}}(E).
\]

\begin{lemma}\label{thm:neighbourhood-renewal}
There exist constants $0<c\leq C<\infty$ such that $c\leq Q(t)\leq C$ for every $t\geq0$.
Moreover, for every $i$,
\begin{equation}\label{eq:neighbourhood-renewal-shift}
  Q(t-\alpha_i/2)-Q(t)\to0
  \qquad\text{as }t\to\infty.
\end{equation}
\end{lemma}

\begin{proof}
The bounds $c\leq Q(t)\leq C$ follow from \cref{thm:neighbourhood-moments}. Brownian scaling and
\cref{eq:neighbourhood-union-scaling} give, for every $r>0$,
\[
  \bbE M_r(\BM(S_i(K)))=r_i\bbE M_{r/\sqrt{r_i}}(E).
\]
Recall that $\alpha_i=\log(1/r_i)$ and put $d(t)=e^{(2-2s)t}\bbE D_{e^{-t}}$ for $t\geq0$.
The definition of $D_r$ gives, for $t\geq\frac12\max_i \alpha_i$,
\begin{equation}\label{eq:neighbourhood-renewal-equation}
  \begin{split}
    Q(t)
    &=e^{(2-2s)t}\sum_i\bbE M_{e^{-t}}(\BM(S_i(K)))-d(t)\\
    &=\sum_i r_i e^{(1-s)\alpha_i}Q(t-\alpha_i/2)-d(t)\\
    &=\sum_i p_i Q(t-\alpha_i/2)-d(t).
  \end{split}
\end{equation}
Moreover, \cref{thm:neighbourhood-overlap} gives
\begin{equation}\label{eq:neighbourhood-renewal-defect}
  0\leq d(t)\leq Ce^{(2-2s)t}e^{-(2-2s+2\eta)t}=Ce^{-2\eta t},
  \qquad t\geq0.
\end{equation}

Continuity of $M_r(E)$ in $r$ and the moment bound in \cref{thm:neighbourhood-moments}
give continuity of $Q$ by dominated convergence. We apply the renewal theorem in the form
\cref{eq:renewal-shift} to $Q$ with interarrival distribution $\sum_i p_i\delta_{\alpha_i/2}$,
using \cref{eq:neighbourhood-renewal-equation,eq:neighbourhood-renewal-defect}, to obtain
\cref{eq:neighbourhood-renewal-shift}.
\end{proof}

We next show that the normalised neighbourhood mass concentrates around its mean. For each
finite word $w$, Brownian scaling, using time reversal when $S_w$ reverses orientation, gives
\begin{equation}\label{eq:cylinder-neighbourhood-scaling}
  \bigl(M_r(\BM(S_wK))\bigr)_{r>0}
  \overset{\mathrm d}{=}
  \bigl(r_wM_{r/\sqrt{r_w}}(E)\bigr)_{r>0}.
\end{equation}
By translation invariance, $M_r(\BM(S_wK))$ depends only on the Brownian increments within
$I_w$. The interval partition from \cref{sec:osc-geometry} therefore divides the stopping words
into at most $M$ classes within which these masses are independent.

\begin{lemma}\label{thm:neighbourhood-concentration}
There exist constants $C,\gamma>0$ such that, for every $t\geq0$,
\[
  \Var\bigl(e^{(2-2s)t}M_{e^{-t}}(E)\bigr)\leq Ce^{-\gamma t}.
\]
Consequently, for every fixed $\tau\in\bbR$, we have
\begin{equation}\label{eq:neighbourhood-grid-concentration}
  e^{(2-2s)(n+\tau)}M_{e^{-(n+\tau)}}(E)-Q(n+\tau)\to0
  \qquad\text{almost surely as }n\to\infty.
\end{equation}
\end{lemma}

\begin{proof}
Fix $0<\gamma<\min\{s,2\eta/3\}$ and put $r=e^{-t}$. For $r_{\min}\leq r\leq1$,
\cref{thm:neighbourhood-moments} gives
\[
  \Var\bigl(r^{2s-2}M_r(E)\bigr)\leq r^{4s-4}\bbE[M_r(E)^2]\leq C
  \leq Cr_{\min}^{-\gamma}e^{-\gamma t}.
\]
We may therefore, without loss of generality, assume $0<r<r_{\min}$.
For $w\in\cW(r)$, we have $r_w>r_{\min}r>r^2$, so $r/\sqrt{r_w}<1$.
By \cref{eq:cylinder-neighbourhood-scaling,thm:neighbourhood-moments}, for every $q\geq1$,
\[
  \bbE[M_r(\BM(S_wK))^q]
  \leq C_qr_w^q\left(\frac{r}{\sqrt{r_w}}\right)^{q(2-2s)}
  =C_qr^{q(2-2s)}p_w^q.
\]
As in \cref{sec:osc-geometry}, we partition $\cW(r)$ into classes
$\cW_1(r),\ldots,\cW_M(r)$ such that the time intervals $I_w$ have pairwise disjoint interiors
within each class.
By independence within each class, Cauchy--Schwarz, and $p_w\leq r^s$, we obtain
\begin{equation}\label{eq:stopping-neighbourhood-variance}
  \begin{split}
    r^{4s-4}\Var\left(\sum_{w\in\cW(r)}M_r(\BM(S_wK))\right)
    &\leq Mr^{4s-4}\sum_{j=1}^M
      \Var\left(\sum_{w\in\cW_j(r)}M_r(\BM(S_wK))\right)\\
    &=Mr^{4s-4}\sum_{w\in\cW(r)}\Var(M_r(\BM(S_wK)))\\
    &\leq C\sum_{w\in\cW(r)}p_w^2
     \leq Cr^s\sum_{w\in\cW(r)}p_w=Cr^s.
  \end{split}
\end{equation}

We next bound the total overlap defect. Let $\cT(r)$ be the set of proper prefixes of words
in $\cW(r)$, including the empty word. For $v\in\cT(r)$, put
\[
  D_{v,r}=\sum_i M_r(\BM(S_vS_iK))-M_r(\BM(S_vK))\geq0.
\]
The terms for internal cylinders cancel, giving
\[
  \sum_{v\in\cT(r)}D_{v,r}
  =\sum_{w\in\cW(r)}M_r(\BM(S_wK))-M_r(E).
\]
For $v\in\cT(r)$, we have $r_v>r$. Brownian scaling and
\cref{eq:neighbourhood-defect-mean} give
\begin{align*}
  r^{2s-2}\bbE D_{v,r}
  &=r^{2s-2}r_v\bbE D_{r/\sqrt{r_v}}
  \leq Cr^{2s-2}r_v\left(\frac{r}{\sqrt{r_v}}\right)^{2-2s+2\eta}
   =Cr^{2\eta}p_vr_v^{-\eta}.
\end{align*}
Since $r<r_v\leq(\max_i r_i)^{\abs{v}}$, the set $\cT(r)$ has at most
$C(1+\log(1/r))$ generations. In each generation,
\[
  \sum_{\substack{v\in\cT(r)\\\abs{v}=k}}p_v
  \leq\sum_{\abs{v}=k}p_v=1.
\]
Consequently,
\begin{equation}\label{eq:stopping-neighbourhood-defect}
  \begin{split}
    r^{2s-2}\sum_{v\in\cT(r)}\bbE D_{v,r}
    &\leq Cr^{2\eta}\sum_{v\in\cT(r)}p_vr_v^{-\eta}
     \leq Cr^\eta\sum_{v\in\cT(r)}p_v\\
    &\leq Cr^\eta(1+\log(1/r)).
  \end{split}
\end{equation}
The total defect is nonnegative and at most $\sum_{w\in\cW(r)}M_r(\BM(S_wK))$.
Since $\sum_{w\in\cW(r)}p_w=1$, Jensen's inequality and the fourth-moment bound give
\begin{align*}
  r^{8s-8}\bbE\left[\left(\sum_{v\in\cT(r)}D_{v,r}\right)^4\right]
  &\leq r^{8s-8}\sum_{w\in\cW(r)}p_w^{-3}\bbE[M_r(\BM(S_wK))^4]\\
  &\leq C\sum_{w\in\cW(r)}p_w=C.
\end{align*}
H\"older's inequality and \cref{eq:stopping-neighbourhood-defect} now give
\begin{align*}
  r^{4s-4}\bbE\left[\left(\sum_{v\in\cT(r)}D_{v,r}\right)^2\right]
  &\leq r^{4s-4}\left(\sum_{v\in\cT(r)}\bbE D_{v,r}\right)^{2/3}
    \left(\bbE\left[\left(\sum_{v\in\cT(r)}D_{v,r}\right)^4\right]\right)^{1/3}\\
  &\leq Cr^{2\eta/3}(1+\log(1/r))^{2/3}.
\end{align*}
Subtracting the total defect from the stopping sum gives $M_r(E)$, so
\cref{eq:stopping-neighbourhood-variance} and the last bound yield
\begin{align*}
  \Var\bigl(r^{2s-2}M_r(E)\bigr)
  &\leq2r^{4s-4}\left(
    \Var\left(\sum_{w\in\cW(r)}M_r(\BM(S_wK))\right)
    +\bbE\left[\left(\sum_{v\in\cT(r)}D_{v,r}\right)^2\right]\right)\\
  &\leq Cr^s+Cr^{2\eta/3}(1+\log(1/r))^{2/3}\\
  &=Ce^{-st}+C(1+t)^{2/3}e^{-2\eta t/3}
   \leq Ce^{-\gamma t}.
\end{align*}

Chebyshev's inequality and the Borel--Cantelli lemma, as in the proof of
\cref{thm:uniform-concentration}, give \cref{eq:neighbourhood-grid-concentration} for each fixed
$\tau$ almost surely.
\end{proof}

\subsection{Reconstruction}

The renewal and concentration estimates in
\cref{thm:neighbourhood-renewal,thm:neighbourhood-concentration} now recover the self-similar
cylinder weights from the image set.

\begin{proof}[Proof of \cref{thm:minkowski-reconstruction}]
For each finite word $w$, put $\alpha_w=\log(1/r_w)$. By the equality in law of processes in
\cref{eq:cylinder-neighbourhood-scaling}, applying the almost-sure convergence
\cref{eq:neighbourhood-grid-concentration} with the fixed offset $-\alpha_w/2$ and taking a
countable intersection gives, simultaneously for every word $w$,
\[
  e^{(2-2s)n}M_{e^{-n}}(\BM(S_wK))-p_w Q(n-\alpha_w/2)\to0
  \qquad\text{almost surely as }n\to\infty.
\]

Since $\alpha_w$ is the sum of the logarithmic contractions along $w$, repeated application
of the mean shift relation \cref{eq:neighbourhood-renewal-shift} gives
$Q(n-\alpha_w/2)-Q(n)\to0$. Combining this with \cref{eq:neighbourhood-grid-concentration} at
$\tau=0$ and the lower bound in \cref{thm:neighbourhood-renewal}, we obtain, simultaneously
for every word $w$,
\begin{equation}\label{eq:cylinder-neighbourhood-ratio}
  \frac{M_{e^{-n}}(\BM(S_w(K)))}{M_{e^{-n}}(E)}
  =\frac{p_w Q(n-\alpha_w/2)+o(1)}{Q(n)+o(1)}\to p_w
  \qquad\text{almost surely as }n\to\infty.
\end{equation}

Fix an integer $k\geq1$. Since $E=\bigcup_{\abs{w}=k}\BM(S_w(K))$ and
$\sum_{\abs{w}=k}p_w=1$, \cref{eq:cylinder-neighbourhood-ratio} gives
\begin{equation}\label{eq:level-neighbourhood-defect}
  \begin{split}
    \frac{\sum_{\abs{w}=k}M_{e^{-n}}(\BM(S_w(K)))-M_{e^{-n}}(E)}{M_{e^{-n}}(E)}
    &=\sum_{\abs{w}=k}\frac{M_{e^{-n}}(\BM(S_w(K)))}{M_{e^{-n}}(E)}-1\\
    &\to0
    \qquad\text{as }n\to\infty.
  \end{split}
\end{equation}

Recall that $\nu_{E,e^{-n}}$ is the cut-off probability measure defined in
\cref{eq:neighbourhood-measure-definition}. Let $f\colon\bbR^2\to\bbR$ be bounded and Lipschitz.
For every word $w$ of length $k$, choose $t_w\in S_w(K)$ and put
$x_w\coloneqq\BM_{t_w}$. We claim that
\[
  \limsup_{n\to\infty}
  \abs{\int f\dd\nu_{E,e^{-n}}-\sum_{\abs{w}=k}p_wf(x_w)}
  \leq
  \Lip(f)\max_{\abs{w}=k}\diam(\BM(S_w(K))).
\]

To prove the claim, we first replace the maximum of the cylinder cut-offs by their sum. The
resulting
error is at most $\norm{f}_\infty$ times the ratio in
\cref{eq:level-neighbourhood-defect}. Moreover, if
$x\in\langle\BM(S_w(K))\rangle_{e^{-n}}$, then
\[
  \abs{f(x)-f(x_w)}
  \leq\Lip(f)(\diam(\BM(S_w(K)))+e^{-n}).
\]
Integrating against the cylinder cut-offs and dividing by $M_{e^{-n}}(E)$ bounds the error from
replacing $f(x)$ by $f(x_w)$ on each cylinder by
\[
  \Lip(f)
  \sum_{\abs{w}=k}(\diam(\BM(S_w(K)))+e^{-n})
  \frac{M_{e^{-n}}(\BM(S_w(K)))}{M_{e^{-n}}(E)}.
\]
Finally, \cref{eq:cylinder-neighbourhood-ratio} shows that the additional error incurred by
replacing the remaining cylinder ratios by $p_w$ tends to zero.

As $k\to\infty$, the diameters of the sets $S_w(K)$ tend to zero uniformly over words of length
$k$. The uniform continuity of the Brownian path therefore gives
\[
  \max_{\abs{w}=k}\diam(\BM(S_w(K)))\to0.
\]
Iterating self-similarity gives $\mu=\sum_{\abs{w}=k}p_w(S_w)_*\mu$. The same
uniform-continuity estimate therefore gives
\[
  \sum_{\abs{w}=k}p_wf(x_w)\to\int f\dd(\BM_*\mu).
\]
Letting first $n\to\infty$ and then $k\to\infty$ proves the weak convergence
\cref{eq:minkowski-reconstruction}.

Finally, $F\mapsto\nu_{F,r}$ is continuous for each $r>0$ by dominated convergence. Since
$\cP(\bbR^2)$ is Polish, the limit $\lim_{n\to\infty}\nu_{F,e^{-n}}$ is Borel measurable on
its Borel set of convergence. The almost-sure convergence above therefore proves the Borel
assertion.
\end{proof}

\section{Examples and limitations of profile separation}\label{sec:examples}

In this section we give two explicit examples of self-similar sets to illustrate the usefulness of
profile separation. The first example in \cref{sec:base-five} gives two different missing digit
sets, where distinct structure gives rise to distinct profiles. 
In \cref{sec:obstruction}, we give an example of non-isometric sets which have identical
correlation integrals, meaning that these Cantor sets and measures cannot be distinguished by the
profiles.

\subsection{Two base-five missing-digit sets with different structure}\label{sec:base-five}

Set $D_1=\{0,1,4\}$ and $D_2=\{0,2,4\}$. For $j\in\{1,2\}$, let $K_j$ be the attractor of
\[
  S_d(x)=\frac{x+d}{5},\qquad d\in D_j,
\]
and let $\mu_j$ be the self-similar probability measure assigning weight $1/3$ to each map.
Both IFSs satisfy the open set condition with $U=(0,1)$, and both measures are
$s$-Ahlfors regular for $s=\log3/\log5$~\cite{Falconer1997}. The pieces
$S_0(K_1)$ and $S_1(K_1)$ meet at $1/5$, whereas the first-level pieces of $K_2$ are
separated by gaps of length $1/5$.

Set $\ell=\log5$ and write
\(
G_j(t)=e^{st}\cC_{e^{-t}}(\mu_j)\text{ for } j\in\{1,2\}.
\)

\begin{proposition}\label{thm:base-five-profiles}
For every integer $n\geq1$,
\begin{equation}\label{eq:base-five-correlation}
  \cC_{5^{-n}}(\mu_1)=\frac32\,3^{-n}-\frac12\,9^{-n},
  \qquad \cC_{5^{-n}}(\mu_2)=3^{-n}.
\end{equation}
In particular, the expected Brownian correlation profiles satisfy
\[
  \limsup_{t\to\infty}\abs{H_{\mu_1}^s(t)-H_{\mu_2}^s(t)}>0.
\]
\end{proposition}

\begin{proof}
For $n\geq1$, each level-$n$ cylinder interval has length
$5^{-n}$ and mass $3^{-n}$. Pairs in the same cylinder contribute $3^{-n}$ to the
correlation integral. Distinct cylinder intervals either touch or are separated by at least
$5^{-n}$, with equality contributing zero since the measures are atomless. For two touching
intervals, the distance is less than $5^{-n}$ precisely when the rescaled right point is smaller
than the rescaled left point. By exchangeability, each order therefore contributes
$9^{-n}/2$.

For $K_1$, touching pairs have words with a common prefix of length $k$, followed by $0$ and
$1$, and then respectively $n-k-1$ copies of $4$ and $0$. There are consequently
$\sum_{k=0}^{n-1}3^k=(3^n-1)/2$ such pairs, whereas $K_2$ has none. Adding these
contributions proves \cref{eq:base-five-correlation}.

Both IFSs are of lattice type with span $\ell$, so the renewal limits in \cref{thm:renewal-limits}
give continuous $\ell$-periodic functions $\widetilde G_j$ with
$G_j(t)-\widetilde G_j(t)\to0$ as $t\to\infty$, and
\cref{eq:base-five-correlation} gives
\[
  \widetilde G_1(0)=\lim_{n\to\infty}G_1(n\ell)=\frac32,
  \qquad
  \widetilde G_2(0)=\lim_{n\to\infty}G_2(n\ell)=1.
\]
The expected profiles converge to their smoothed periodic limits, as in the proof of
\cref{thm:profile-asymptotics}. These limits are distinct by the injectivity of smoothing in
\cref{thm:smoothing-injective}, giving the separation.
\end{proof}

Applying \cref{thm:main} and then the Borel reconstruction map from
\cref{thm:minkowski-reconstruction}, we obtain
\[
  \Law(\BM_*\mu_1)\perp\Law(\BM_*\mu_2),
  \qquad \Law(\BM(K_1))\perp\Law(\BM(K_2)).
\]
Thus one compact planar Brownian image almost surely distinguishes these two missing-digit sets.

\subsection{Identical correlation integrals}\label{sec:obstruction}

The expected profile does not determine the original measure. For a probability measure $\mu$ on
$[0,1]$, let $\widetilde\mu$ be its pushforward under $x\mapsto-x$. If $X$ and $Y$ are independent
with law $\mu$, then $\mu*\widetilde\mu$ is the law of $X-Y$ and is called the
\emph{autocorrelation measure} of $\mu$. Two probability measures on $[0,1]$ are
\emph{homometric} if their autocorrelation measures agree.
Rosenblatt--Seymour~\cite{RosenblattSeymour1982} studied the structure of homometric sets and finite
measures. Since $\cC_r(\mu)=(\mu*\widetilde\mu)((-r,r))$ for every $r>0$, homometric measures
have identical correlation integrals and, by \cref{thm:gaussian-reduction}, the same expected
Brownian correlation integral at every radius.

We construct such a pair at dimension $s=\log6/\log30\in(1/2,1)$.
We take the digit sets
\[
  D_A\coloneqq\set{0,1,4,10,12,17}
  \quad\text{and}\quad
  D_B\coloneqq\set{0,1,8,11,13,17}.
\]
Let $\mu_{D_A}$ and $\mu_{D_B}$ be the equal-weight self-similar probability measures on $[0,1]$
for the respective digit sets, with maps
\[
  S_d(x)\coloneqq\frac{x}{30}+\frac d{18}.
\]

\begin{proposition}\label{thm:homometric-example}
The measures $\mu_{D_A}$ and $\mu_{D_B}$ are distinct and $s$-Ahlfors regular. Their defining
IFSs are strongly separated, their supports are not isometric, and
\[
  \mu_{D_A}*\widetilde{\mu_{D_A}}=\mu_{D_B}*\widetilde{\mu_{D_B}}.
\]
\end{proposition}

\begin{proof}
For every $k\in\bbZ$, direct calculation gives
\[
  \begin{aligned}
    \#\set{(d,e)\in D_A\times D_A\colon d-e=k}
    &=\#\set{(d,e)\in D_B\times D_B\colon d-e=k}\\
    &=\begin{cases}
      6,&k=0,\\
      1,&\abs k\in\set{1,\ldots,13,16,17},\\
      0,&\text{otherwise}.
    \end{cases}
  \end{aligned}
\]

Both supports have convex hull $[0,85/87]$, so any isometry between them is either the identity
or the reflection $x\mapsto85/87-x$. Reflection replaces the digit set $D_A$ by $17-D_A$.
Since the first-level intervals are pairwise disjoint, equality of attractors forces equality
of their digit sets. Neither $D_A$ nor $17-D_A$ equals $D_B$, proving non-isometry and hence
$\mu_{D_A}\neq\mu_{D_B}$.

Writing the self-similar measures as random digit expansions, their autocorrelations are the
laws of
\[
  \frac1{18}\sum_{k=0}^\infty\frac{d_k-e_k}{30^k},
\]
where all digits $d_k,e_k$ are independent and uniform on the respective digit set. The signed
difference multisets agree, so the two series have the same law.
\end{proof}

By \cref{thm:homometric-example,eq:expected-profile}, $H_{\mu_{D_A}}^s=H_{\mu_{D_B}}^s$, so the
two-point criterion of \cref{thm:main} cannot separate $\Law(\BM_*\mu_{D_A})$ from
$\Law(\BM_*\mu_{D_B})$. 

\section{A two-contraction family of fixed dimension}\label{sec:fixed-dimension}

\subsection{Correlation profiles}\label{sec:fixed-dimension-family}

We use the two-contraction family introduced before \cref{thm:fixed-dimension-separation}.
Thus $s\in(0,1)$ is fixed, $p\in(0,1)$, $q=1-p$, and the contractions are
$\lambda_1=p^{1/s}$ and $\lambda_2=q^{1/s}$.

The cross term of \cref{sec:renewal-equation} is
\[
  \cC^\times_\delta(\mu_p)
  =2pq\,(\mu_p\times\mu_p)\set{(x,y)\colon S_2(y)-S_1(x)\leq\delta}.
\]
Here $1-\lambda_1-\lambda_2\leq S_2(y)-S_1(x)\leq1$ for $x,y\in K_p$, and
$\cC^\times_1(\mu_p)=2pq$. For the entropy of $\mu_p$, we write $h(p)=-p\log p-q\log q$.
Reflection in $1/2$ interchanges $p$ and $1-p$, preserving distances, so
\[
  \cC^\times_\delta(\mu_p)=\cC^\times_\delta(\mu_{1-p}),
  \qquad\overline H_s(p)=\overline H_s(1-p).
\]
We therefore restrict to $(0,1/2]$ where convenient.

\begin{proposition}\label{thm:fixed-dimension-profiles}
For every $p\in(0,1)$ and $r>0$,
\begin{equation}\label{eq:fixed-dimension-correlation}
    \bbE\cC_r(\BM_*\mu_p)
    =\sum_{j,k\geq0}\binom{j+k}{j}p^{2j}q^{2k}
    \cdot\int_0^1\left(1-\exp\left(-\frac{r^2}{2\lambda_1^j\lambda_2^k\delta}\right)\right)
      \dd\cC^\times_\delta(\mu_p).
\end{equation}
The mean profile exists and satisfies
\begin{equation}\label{eq:fixed-dimension-mean}
  \overline H_s(p)=\frac{2^{-s}\Gamma(1-s)}{h(p)}
    \int_0^1\delta^{-s}\dd\cC^\times_\delta(\mu_p).
\end{equation}
If $\log\lambda_1/\log\lambda_2\notin\bbQ$, then $H_{\mu_p}^s(t)\to\overline H_s(p)$ as
$t\to\infty$.
\end{proposition}

\begin{proof}
Iterating the correlation recursion \cref{eq:correlation-recursion} gives
\[
  \cC_\delta(\mu_p)
  =\sum_{j,k\geq0}\binom{j+k}{j}p^{2j}q^{2k}
    \cC^\times_{\delta/(\lambda_1^j\lambda_2^k)}(\mu_p).
\]
The Gaussian reduction in \cref{eq:gaussian-reduction} and Fubini's theorem now give
\cref{eq:fixed-dimension-correlation}.

For the renewal equation in \cref{sec:renewal-equation}, the mean interarrival time is
$m=h(p)/s$ and the inhomogeneous term is $z(t)=e^{st}\cC^\times_{e^{-t}}(\mu_p)$.
By Fubini's theorem,
\[
  \int_\bbR z(t)\dd t
  =\int_0^1\int_{-\infty}^{-\log\delta}e^{st}\dd t\,\dd\cC^\times_\delta(\mu_p)
  =\frac1s\int_0^1\delta^{-s}\dd\cC^\times_\delta(\mu_p).
\]
The renewal limits in \cref{thm:renewal-limits} and the lattice periodisation in
\cref{eq:lattice-periodisation} show that $G$ has Ces\`aro mean $m^{-1}\int_\bbR z$ in both
cases. Averaging the smoothing formula \cref{eq:h-definition} and applying dominated
convergence multiplies this mean by
the zero-frequency multiplier in
\cref{eq:gamma-multiplier}, that is $\int_0^\infty\phi(\eta)\dd\eta=2^{-s}\Gamma(1-s)$.
Substitution gives \cref{eq:fixed-dimension-mean}, and ordinary
convergence in the non-lattice case follows from the profile asymptotics in
\cref{thm:profile-asymptotics}.
\end{proof}

\subsection{Separation of image laws}\label{sec:fixed-dimension-separation}

We first show that, for fixed $s\in(0,1)$, the mean $\overline H_s$ extends holomorphically
to a complex neighbourhood of $(0,1)$ and tends to zero as $p\downarrow0$.
\Cref{thm:fixed-dimension-separation} then draws the consequences for its critical points and level sets.

We replace $p$ by a complex number $z\in\bbC\setminus((-\infty,0]\cup[1,\infty))$ and define
$\lambda_1(z)=z^{1/s}$ and $\lambda_2(z)=(1-z)^{1/s}$ using the principal branches. We put
\[
  S_{1,z}(x)=\lambda_1(z)x,\qquad S_{2,z}(x)=\lambda_2(z)x+1-\lambda_2(z).
\]
For a finite word $w=w_1\cdots w_n\in\{1,2\}^n$, let
\[
  x_w(z)=S_{w_1,z}\circ\cdots\circ S_{w_n,z}(0),
  \qquad
  P_w(z)=z^{N_1(w)}(1-z)^{N_2(w)},
\]
where $N_i(w)$ counts the occurrences of $i$ in $w$, with $x_\varnothing=0$ and
$P_\varnothing=1$. If $\abs{\lambda_1},\abs{\lambda_2}\leq\rho<1$, then $\abs{S_{i,z}(0)}\leq2$, and
successive compositions give
\begin{equation}\label{eq:coded-point-bounds}
  \abs{x_{wi}(z)-x_w(z)}\leq2\rho^{n}
  \qquad\text{and}\qquad
  \abs{x_w(z)}\leq\frac{2}{1-\rho}.
\end{equation}
Since $P_1+P_2=1$, the weights of the words of length $n$ sum to $1$, while
$\sum_{w\in\{1,2\}^n}\abs{P_w(z)}=(\abs z+\abs{1-z})^n$. 

For an open disc $U\subset\bbC$, we write $B=\sup_{z\in U}(\abs z+\abs{1-z})$. For a kernel
$\psi\colon U\times D\times D\to\bbC$, with $D\subset\bbC$, we write
\[
  \Lambda^{(m)}(z)=\sum_{u,v\in\{1,2\}^m}P_u(z)P_v(z)\,\psi\bigl(z;x_u(z),x_v(z)\bigr),
  \qquad m\geq0,
\]
for its \emph{level sums}. At a real parameter $p\in(0,1)$, expectations below are taken over
independent $X,Y$ with law $\mu_p$.

Analytic dependence of stationary integrals is studied by
Cipriano--Pollicott~\cite[Theorem~1.2]{CiprianoPollicott2018}. Here we also need the uniform
bound below.

\begin{lemma}\label{thm:level-sums}
Let $U\subset\bbC$ be an open disc on which $\lambda_1$ and $\lambda_2$ are holomorphic and
bounded in modulus by some $\rho<1$, with $\rho B^2<1$. Let $D\subset\bbC$ be a bounded open convex set
containing $x_w(z)$ for every $z\in U$ and every finite word $w$, and let $\psi$ be holomorphic
on $U\times D\times D$ with $\abs{\partial_x\psi},\abs{\partial_y\psi}\leq L$.
Then the level sums converge uniformly on $U$ to a holomorphic function $\Lambda$ with
\begin{equation}\label{eq:level-sum-bound}
  \sup_{z\in U}\abs{\Lambda(z)-\psi(z;0,0)}\leq\frac{4LB^2}{1-\rho B^2},
\end{equation}
and $\Lambda(p)=\bbE\,\psi(p;X,Y)$ for every $p\in U\cap(0,1)$.
\end{lemma}

\begin{proof}
The derivative bound, \cref{eq:coded-point-bounds}, and $P_1+P_2=1$ give
\[
  \sup_{z\in U}\abs{\Lambda^{(m+1)}(z)-\Lambda^{(m)}(z)}
  \leq4LB^2(\rho B^2)^m.
\]
Summing from $\Lambda^{(0)}=\psi(z;0,0)$ gives uniform convergence and
\cref{eq:level-sum-bound}. The limit is holomorphic by locally uniform convergence.
For $p\in U\cap(0,1)$, dominated convergence of the finite coding approximations gives the
expectation identity.
\end{proof}

\begin{lemma}\label{thm:fixed-dimension-analyticity}
For every fixed $s\in(0,1)$, the functions
\[
  p\mapsto\int_0^1\delta^{-s}\dd\cC^\times_\delta(\mu_p)
  \quad\text{and}\quad p\mapsto\overline H_s(p)
\]
are real analytic on $(0,1)$.
\end{lemma}

\begin{proof}
Fix $p_0\in(0,1)$ and $\rho\in(\max\{\lambda_1(p_0),\lambda_2(p_0)\},1)$. On a small open disc
$U$ centred at $p_0$, the branches $\lambda_1$ and $\lambda_2$ are holomorphic with
$\abs{\lambda_1},\abs{\lambda_2}\leq\rho$, and $B=\sup_U(\abs z+\abs{1-z})$ is as close to $1$
as we wish, so we may assume $\rho B^2<1$.
Iterating the contraction maps gives
$\abs{x_w(z)-x_w(p_0)}\leq C\abs{z-p_0}/(1-\rho)$ uniformly in $w$. Since
$g_0=1-\lambda_1(p_0)-\lambda_2(p_0)>0$, we may shrink $U$ and choose a bounded convex
neighbourhood $D$ of $[0,1]$ containing every $x_w(z)$ such that the base of
\[
  \psi(z;x,y)=\bigl(1-\lambda_2(z)+\lambda_2(z)y-\lambda_1(z)x\bigr)^{-s}
\]
has real part at least $g_0/2$ on $U\times D\times D$. With the principal branch,
$\psi$ is holomorphic there and $\abs{\partial_x\psi},\abs{\partial_y\psi}\leq s(g_0/2)^{-s-1}$.
The level-sum lemma \cref{thm:level-sums} therefore gives a holomorphic function $\Lambda$
on $U$ with
\[
  \Lambda(p)=\frac1{2pq}\int_0^1\delta^{-s}\dd\cC^\times_\delta(\mu_p)
\]
for real $p\in U$. Multiplication by $2z(1-z)$ extends the cross-term integral holomorphically.
After shrinking $U$, the principal-branch extension $h(z)=-z\log z-(1-z)\log(1-z)$ is nonzero,
since $h(p_0)>0$, so \cref{eq:fixed-dimension-mean} also extends the mean holomorphically.
These local extensions agree on overlapping discs by the identity theorem, giving extensions to
a complex neighbourhood of $(0,1)$. Their restrictions to $(0,1)$ are real analytic.
\end{proof}

To estimate the cross term as $p\downarrow0$, let $N$ be the number of symbols $2$ before the
first symbol $1$ in the coding of $Y$. Then $\bbP(N=n)=pq^n$ for $n\geq0$, and
$Y=1-\lambda_2^N+\lambda_1\lambda_2^NY''$ with $Y''$ of law $\mu_p$ and independent of $N$.
Therefore
\begin{equation}\label{eq:initial-symbol-decomposition}
  S_2(Y)-S_1(X)=1-\lambda_2^{N+1}-\lambda_1\bigl(X-\lambda_2^{N+1}Y''\bigr),
\end{equation}
with $N$, $X$ and $Y''$ independent. We write
\[
  J(p)\coloneqq\frac1{2pq}\int_0^1\delta^{-s}\dd\cC^\times_\delta(\mu_p)
\]
for the normalised cross-term integral and put $T(p)=\log(1/\lambda_2)=\frac1s\log(1/q)$.
Omitting $\lambda_1(X-\lambda_2^{N+1}Y'')$ from the distance in
\cref{eq:initial-symbol-decomposition} gives the \emph{main term}
\begin{equation}\label{eq:main-term}
  J_0(p)\coloneqq\sum_{n\geq0}pq^n\bigl(1-\lambda_2^{n+1}\bigr)^{-s}
  =\bigl(e^{sT(p)}-1\bigr)\sum_{m\geq1}\bigl(e^{mT(p)}-1\bigr)^{-s}.
\end{equation}

\begin{proposition}\label{thm:fixed-dimension-range}
For every fixed $s\in(0,1)$, the function $p\mapsto\overline H_s(p)$ is continuous and strictly
positive on $(0,1/2]$, and
\[
  \lim_{p\downarrow0}\overline H_s(p)=0.
\]
\end{proposition}

\begin{proof}
Continuity follows from \cref{thm:fixed-dimension-analyticity}, and positivity from
\cref{thm:fixed-dimension-profiles}, since the cross term has total mass $2pq>0$.

Since $1-\lambda_2\sim p/s$ and $\lambda_1=p^{1/s}=o(p)$ as $p\downarrow0$, we have
$\lambda_1\leq(1-\lambda_2)/2$ for all sufficiently small $p$, and then
\cref{eq:initial-symbol-decomposition} gives
\[
  S_2(Y)-S_1(X)\geq1-\lambda_2^{N+1}-\lambda_1\geq\tfrac12(1-\lambda_2^{N+1}).
\]
Hence, by \cref{eq:main-term},
\[
  J(p)\leq2^sJ_0(p)
  \leq2^s\,\frac{e^{sT(p)}-1}{T(p)}\int_0^\infty(e^u-1)^{-s}\dd u,
\]
where the last step compares the sum with the integral of the decreasing integrand. The
integral is finite because $0<s<1$, and $(e^{sT(p)}-1)/T(p)\to s$ as $p\downarrow0$, so the
cross-term integral is $O(p)$. Since $h(p)\sim p\log(1/p)$, the mean formula
\cref{eq:fixed-dimension-mean} gives
\[
  \overline H_s(p)=O\!\left(\frac{p}{h(p)}\right)
  =O\!\left(\frac1{\log(1/p)}\right)\longrightarrow0
  \qquad(p\downarrow0).\qedhere
\]
\end{proof}

\begin{proof}[Proof of \cref{thm:fixed-dimension-separation}]
By \cref{thm:fixed-dimension-analyticity,thm:fixed-dimension-range}, the mean is non-constant
and real analytic on $(0,1)$. Its derivative and each difference
$\overline H_s-\overline H_s(p_1)$ therefore have isolated zeros. This proves that $D_s$ has
no accumulation point in $(0,1/2]$. On each component of $(0,1/2)\setminus D_s$, the derivative
has constant sign, so the mean is strictly monotone. For fixed $p_1\in(0,1/2]$, positivity of
$\overline H_s(p_1)$ and the limit $\overline H_s(p)\to0$ as $p\downarrow0$ keep its level
set bounded away from zero. It is therefore contained in a compact subinterval of $(0,1)$
and is finite, since its points are isolated zeros of an analytic function.

Finally, let $p_1,p_2\in(0,1/2]$ have distinct means. Since the profiles are bounded and have
these Ces\`aro means,
\[
  \limsup_{t\to\infty}\abs{H_{\mu_{p_1}}^s(t)-H_{\mu_{p_2}}^s(t)}
  \geq\abs{\overline H_s(p_1)-\overline H_s(p_2)}>0,
\]
and the occupation-law criterion \cref{thm:main} gives mutual singularity.
\end{proof}

\subsection{Monotonicity for small dimension}\label{sec:small-dimension}

By \cref{thm:fixed-dimension-analyticity}, $J$ is real analytic, and $J(p)\geq1$ since the
cross term has total mass $2pq$ and is supported on $(0,1]$. We put
\[
  \kappa(p)\coloneqq\frac{d}{dp}\log\frac{pq}{h(p)}
  =\frac{q^2\log(1/q)-p^2\log(1/p)}{pq\,h(p)}.
\]
By \cref{eq:fixed-dimension-mean},
\begin{equation}\label{eq:log-derivative-mean}
  \frac{d}{dp}\log\overline H_s(p)=\kappa(p)+\frac{J'(p)}{J(p)}.
\end{equation}
We first quantify the strict increase of the entropy factor $pq/h(p)$ on $(0,1/2)$.

\begin{lemma}\label{thm:entropy-factor}
The function $\kappa$ is strictly positive on $(0,1/2)$. For $0<p\leq1/50$,
\[
  \kappa(p)\geq\frac{0.88}{p(1-\log p)},
\]
and for every $p_1\in(0,1/2)$ there is $c>0$ such that $\kappa(p)\geq c(1/2-p)$ for
$p_1\leq p<1/2$.
\end{lemma}

\begin{proof}
Put $\Phi(p)=q^2\log(1/q)-p^2\log(1/p)$, so that $\kappa=\Phi/(pqh)$. Then
$\Phi(0+)=\Phi(1/2)=0$, $\Phi'(p)=1-2h(p)$ and $\Phi''(p)=-2h'(p)<0$ on $(0,1/2)$, so $\Phi$
is strictly concave and therefore positive on $(0,1/2)$, and $\kappa>0$.

For the second claim, we use $\log(1/q)\geq p$ and $q\log(1/q)\leq p$, which give
$\Phi(p)\geq p\bigl(q^2-p\log(1/p)\bigr)$ and $h(p)\leq p(1-\log p)$. For $p\leq1/50$ we have
$q^2\geq0.96$ and $p\log(1/p)\leq0.08$, so
\[
  \kappa(p)\geq\frac{q^2-p\log(1/p)}{q\,p(1-\log p)}\geq\frac{0.88}{p(1-\log p)}.
\]

The final claim follows from concavity,
$\Phi(p)\geq\Phi(p_1)(1/2-p)/(1/2-p_1)$, and $pqh(p)\leq\frac14\log2$.
\end{proof}

We first bound the derivative of the main term in \cref{eq:main-term}.

\begin{lemma}\label{thm:main-term-derivative}
Let $0<s\leq1/2$ and $0<p\leq1/50$. Then $\abs{J_0'(p)}\leq6p^{-s}$.
\end{lemma}

\begin{proof}
Put $f(x)=(e^x-1)^{-s}$ and $g(x)=xf'(x)$ for $x>0$. Since $f$ is positive, decreasing
and convex, comparison on the intervals $[mT(p),(m+1)T(p)]$ and integration by parts give
\[
  0\leq T(p)\sum_{m\geq1}f(mT(p))-\int_{T(p)}^\infty f\leq T(p)f(T(p)),
\]
and
\[
  \Bigl|T(p)\sum_{m\geq1}g(mT(p))-\int_{T(p)}^\infty g\Bigr|
  \leq T(p)\int_{T(p)}^\infty\abs{g'}
  \leq T(p)\bigl(2f(T(p))+\abs{g(T(p))}\bigr).
\]
Using $\int_{T(p)}^\infty g=-T(p)f(T(p))-\int_{T(p)}^\infty f$, 
we obtain
\begin{align*}
  \frac{\abs{J_0'(p)}}{T'(p)}
  &\leq\Bigl(\frac{se^{sT(p)}}{T(p)}-\frac{e^{sT(p)}-1}{T(p)^2}\Bigr)
    \int_{T(p)}^\infty f
  +se^{sT(p)}\bigl(3f(T(p))+\abs{g(T(p))}\bigr)\\
  &\leq se^{sT(p)}\Bigl(\frac s2\int_{T(p)}^\infty f+3f(T(p))+\abs{g(T(p))}\Bigr),
\end{align*}
where we used $0\leq xe^x-e^x+1\leq x^2e^x/2$ for $x\geq0$.

The bounds $f(x)\leq x^{-s}$ for $x>0$ and $f(x)\leq(1-e^{-1})^{-s}e^{-sx}$ for $x\geq1$
give $s\int_{T(p)}^\infty f\leq s/(1-s)+1.26\leq2.3$. Also,
$\abs{g(T(p))}\leq s(1+T(p))T(p)^{-s}$, since $xe^x\leq(1+x)(e^x-1)$ for $x\geq0$.
As $p/s\leq T(p)\leq p/(sq)$, we have $s(1+T(p))\leq s+p/q\leq1$ and
$T(p)^{-s}\leq p^{-s}$. Finally, $e^{sT(p)}=1/q$ and $T'(p)=1/(sq)$ give
\[
  \abs{J_0'(p)}
  \leq\frac{1}{q^2}\bigl(1.15+4p^{-s}\bigr)\leq6p^{-s}.\qedhere
\]
\end{proof}

The difference $J(p)-J_0(p)$ is small together with its derivative. We bound the derivative
through a holomorphic extension to a disc whose radius is comparable to $p$.

\begin{lemma}\label{thm:remainder-derivative}
Let $0<s\leq1/2$ and $0<p\leq1/50$. Then
\[
  \abs{J'(p)-J_0'(p)}\leq3000\,s\,(3p/2)^{1/s}p^{-2-s}.
\]
\end{lemma}

\begin{proof}
Let $U=\set{z\in\bbC\colon\abs{z-p}<p/2}$. On $U$ we have $\Re z\geq p/2$,
$\abs z\leq3p/2$ and $1-3p/2\leq\abs{1-z}\leq1-p/2$. For the principal branches $\lambda_1$
and $\lambda_2$, put
\[
  \lambda_{1,\max}\coloneqq\sup_U\abs{\lambda_1}=(3p/2)^{1/s},
  \qquad
  \lambda_{2,\max}\coloneqq\sup_U\abs{\lambda_2}=(1-p/2)^{1/s}\leq e^{-p/(2s)}.
\]
Since $\lambda_{1,\max}\leq(3p/2)^2$, $1-\lambda_{2,\max}\geq p/2$ and $p\leq1/50$, we have
\begin{equation}\label{eq:contraction-modulus-bounds}
  5\lambda_{1,\max}\leq p/4\leq\tfrac12(1-\lambda_{2,\max}).
\end{equation}
Applying $\sqrt{1+x}\leq1+x/2$ to the two moduli gives
\[
  \abs z+\abs{1-z}\leq1+\frac{\abs{z-p}^2}{2pq}
  \leq1+\frac p{8q}\leq1+\frac p5.
\]
Thus $B\leq1+p/5$, and $1-e^{-x}\geq\frac12\min(x,1)$ for $x\geq0$ gives
\begin{equation}\label{eq:contraction-disc}
  \lambda_{2,\max} B^2\leq e^{-p/(2s)+2p/5}\leq e^{-3p/(10s)},
  \qquad
  \xi\coloneqq1-\lambda_{2,\max} B^2\geq\tfrac12\min\Bigl(\frac{3p}{10s},1\Bigr)\geq\frac p4,
\end{equation}
using $s\leq1/2$ and $p\leq1/50$ in the last step.

For a finite word $w$ and $z\in U$, writing $w=2^k1w'$ or $w=2^k$ gives
$x_w=1-\lambda_2^k+\lambda_1\lambda_2^kx_{w'}$ or $x_w=1-\lambda_2^k$, so
$\abs{x_w}\leq2+\lambda_{1,\max}\abs{x_{w'}}$, and induction on the length gives
$\abs{x_w(z)}<2/(1-\lambda_{1,\max})\leq5/2$. 

Let
$D=\set{x\in\bbC\colon\abs x<5/2}$.
For $n\geq0$, $z\in U$ and $x,y\in D$, define
\[
  \psi_n(z;x,y)=\bigl(1-\lambda_2^{n+1}-\lambda_1(x-\lambda_2^{n+1}y)\bigr)^{-s}.
\]
By \cref{eq:contraction-modulus-bounds}, the real part of the base is at least
$1-\lambda_{2,\max}^{n+1}-5\lambda_{1,\max}\geq\frac12(1-\lambda_{2,\max}^{n+1})>0$, so $\psi_n$ is holomorphic on
$U\times D\times D$, and
\begin{equation}\label{eq:kernel-bounds}
  \abs{\psi_n}\leq\bigl(\tfrac12(1-\lambda_{2,\max}^{n+1})\bigr)^{-s},
  \qquad
  \abs{\partial_x\psi_n},\abs{\partial_y\psi_n}
  \leq s\lambda_{1,\max}\bigl(\tfrac12(1-\lambda_{2,\max}^{n+1})\bigr)^{-s-1}.
\end{equation}
The level-sum lemma, \cref{thm:level-sums},
applies with $\rho=\lambda_{2,\max}\geq\lambda_{1,\max}$ and the bounds
in \cref{eq:kernel-bounds}, giving a holomorphic limit $\Lambda_n$ and
\begin{equation}\label{eq:lambda-bound}
  \bigl|\Lambda_n(z)-\bigl(1-\lambda_2(z)^{n+1}\bigr)^{-s}\bigr|
  \leq\frac{12\,s\lambda_{1,\max}}{\xi}\bigl(1-\lambda_{2,\max}^{n+1}\bigr)^{-s-1}.
\end{equation}
We define
\[
  \Delta(z)=\sum_{n\geq0}z(1-z)^n
  \Bigl[\Lambda_n(z)-\bigl(1-\lambda_2(z)^{n+1}\bigr)^{-s}\Bigr].
\]
By \cref{eq:lambda-bound} and $\abs{1-z}\leq1-p/2$, the series converges uniformly on $U$, so
$\Delta$ is holomorphic. For real $z\in U$, the expectation identity in \cref{thm:level-sums}
and the decomposition \cref{eq:initial-symbol-decomposition} at the first symbol $1$ give
$\Delta(z)=J(z)-J_0(z)$. Moreover,
\[
  \sup_U\abs\Delta\leq\frac{12\,s\lambda_{1,\max}}{\xi}\sum_{n\geq0}\frac{3p}{2}(1-p/2)^n
  \bigl(1-\lambda_{2,\max}^{n+1}\bigr)^{-s-1}.
\]
Since $\lambda_{2,\max}^{n+1}\leq e^{-(n+1)p/(2s)}$ and $1-e^{-x}\geq\frac12\min(x,1)$,
\[
  \bigl(1-\lambda_{2,\max}^{n+1}\bigr)^{-s-1}
  \leq2^{s+1}\Bigl(1+\Bigl(\frac{2s}{(n+1)p}\Bigr)^{s+1}\Bigr).
\]
Together with $\sum_{n\geq0}(n+1)^{-s-1}\leq1+1/s$, this bounds the sum by
\[
  3\cdot2^{s+1}+\frac32\,4^{s+1}s^s(s+1)\,p^{-s}
  \leq27p^{-s},
\]
using $p^{-s}\geq1$ and $s\leq1/2$. With $1/\xi\leq4/p$ from \cref{eq:contraction-disc},
we obtain $\sup_U\abs\Delta\leq1296\,s\lambda_{1,\max} p^{-1-s}$. Cauchy's estimate now gives
\[
  \abs{J'(p)-J_0'(p)}\leq\frac2p\sup_U\abs\Delta
  \leq3000\,s\lambda_{1,\max} p^{-2-s}.\qedhere
\]
\end{proof}

For parameters away from zero, we bound the second derivative on a disc of fixed radius.
For $s>0$, we put $\gamma=(1-5\cdot10^{-5})^{1/s}$.

\begin{lemma}\label{thm:second-derivative-small-s}
Let $0<s\leq2\cdot10^{-5}$ and $10^{-4}\leq p\leq1/2$. Then
$\abs{J''(p)}\leq10^{11}s\gamma$.
\end{lemma}

\begin{proof}
Put $r=5\cdot10^{-5}$ and $U=\set{z\in\bbC\colon\abs{z-p}<r}$. On $U$ we have
$\abs z\leq1/2+r$ and $\abs{1-z}\leq1-r$, both with positive real parts, so the principal
branches satisfy $\abs{\lambda_1(z)}\leq(1/2+r)^{1/s}\leq\gamma$ and
$\abs{\lambda_2(z)}\leq\gamma$. The bound on $s$ gives $\gamma\leq1/12$, and
$B\coloneqq\sup_U(\abs z+\abs{1-z})\leq1+2r$, so
$\gamma B^2<1$. By \cref{eq:coded-point-bounds}, every coded point satisfies
$\abs{x_w(z)}\leq2/(1-\gamma)<5/2$, and we write $D=\set{x\in\bbC\colon\abs x<5/2}$.

Put $\psi(z;x,y)=(1-\lambda_2+\lambda_2 y-\lambda_1 x)^{-s}$ on $U\times D\times D$. Its base
is within $6\gamma\leq1/2$ of $1$. The bound $\abs{(\zeta^{-s})'}\leq3s$ on
$\abs{\zeta-1}\leq1/2$ therefore gives
\[
  \abs{\psi-1}\leq18s\gamma,
  \qquad
  \abs{\partial_x\psi},\abs{\partial_y\psi}\leq3s\gamma.
\]
The level-sum lemma \cref{thm:level-sums}, with $\rho=\gamma$ and $L=3s\gamma$, gives a
holomorphic function $\Lambda$ equal to $J(z)$ for real $z\in U$ and satisfying
\[
  \sup_U\abs{\Lambda-1}
  \leq18s\gamma+\frac{12B^2s\gamma}{1-\gamma B^2}\leq32s\gamma.
\]
Cauchy's estimate gives $\abs{J''(p)}\leq2\cdot32s\gamma/r^2\leq10^{11}s\gamma$.
\end{proof}

For $0<s\leq1/2$ and $0<p\leq1/50$,
\cref{thm:entropy-factor,thm:main-term-derivative,thm:remainder-derivative} combine to
\begin{equation}\label{eq:small-p-margin}
  p(1-\log p)\bigl(\kappa(p)-\abs{J'(p)}\bigr)
  \geq0.88-\Bigl(6p^{1-s}+3000\,s\,(3p/2)^{1/s}p^{-1-s}\Bigr)(1-\log p).
\end{equation}

\begin{proof}[Proof of \cref{thm:small-dimension-monotone}]
Let $0<s\leq2\cdot10^{-5}$. By \cref{eq:log-derivative-mean} and $J(p)\geq1$, it suffices to
prove $\kappa(p)>\abs{J'(p)}$ for $0<p<1/2$.

For $0<p\leq10^{-4}$, using $(3p/2)^{1/s}\leq(3p/2)^4$ and $s\leq1/2$, the subtracted term
in \cref{eq:small-p-margin} is at most $(6p^{1/2}+2\cdot10^4p^{5/2})(1-\log p)$. This increases
on $(0,1/e)$ and is less than $0.62$ at $p=10^{-4}$, proving $\kappa(p)>\abs{J'(p)}$ throughout
$(0,10^{-4}]$.

For $10^{-4}\leq p<1/2$, the symmetry $J(p)=J(1-p)$ gives $J'(1/2)=0$, so
\cref{thm:second-derivative-small-s} yields $\abs{J'(p)}\leq10^{11}s\gamma\,(1/2-p)$. Let $c$
be the constant of \cref{thm:entropy-factor} for $p_1=10^{-4}$. Since $s\gamma\to0$ as
$s\downarrow0$, there is $s_0\leq2\cdot10^{-5}$ such that $10^{11}s\gamma<c$ for all
$s\leq s_0$, and then $\kappa(p)>\abs{J'(p)}$ on $[10^{-4},1/2)$.

Thus the mean is strictly increasing on $(0,1/2]$ for $s\leq s_0$, and
\cref{thm:fixed-dimension-separation} gives mutual singularity.
\end{proof}

\begin{proof}[Proof of \cref{thm:critical-set-finite}]
For each fixed $0<s\leq1/2$, the error in
\cref{eq:small-p-margin} tends to zero as $p\downarrow0$, since $1/s-1-s\geq1/2$. Hence
\cref{eq:log-derivative-mean} and $J(p)\geq1$ give a positive derivative on $(0,p_0]$ for
some $p_0>0$. The remaining critical points lie in a compact subinterval of $(0,1/2]$ and
cannot accumulate by \cref{thm:fixed-dimension-separation}, so $D_s$ is finite.
\end{proof}

\end{document}